\documentclass[a4paper]{amsart}
\usepackage[usenames,dvipsnames]{color}
\usepackage{amssymb} 
\usepackage{amsmath}
\usepackage{amsthm}
\usepackage[english]{babel}

\allowdisplaybreaks
\newtheorem{theorem}{Theorem}[section]
\newtheorem{lem}[theorem]{Lemma}
\newtheorem{rem}[theorem]{Remark}
\newtheorem{defi}[theorem]{Definition}

\newtheorem{corollary}[theorem]{Corollary}

\usepackage{enumerate}   
\usepackage{verbatim}

\usepackage{hyperref}
\hypersetup{
    colorlinks=true,
    citecolor=blue,
    filecolor=black,
    linkcolor=red,
    urlcolor=ForestGreen,
		linktoc=page
}

\numberwithin{equation}{section}

\begin{document}
\title[Convergence of Riemannian $4$-manifolds]{Convergence of Riemannian $4$-manifolds with \\ $L^2$-curvature bounds}
\author{Norman Zergaenge}

\date{\today.}

\begin{abstract}
In this work we prove convergence results of sequences of Riemannian $4$-manifolds with almost vanishing $L^2$-norm of a curvature tensor and a
non-collapsing bound on the volume of small balls. 

In Theorem \ref{th:1}, we consider a sequence of closed Riemannian $4$-manifolds, whose $L^2$-norm of the 
Riemannian curvature tensor tends to zero. Under the assumption of a uniform non-collapsing bound and a uniform diameter bound, we prove that there exists a subsequence that converges 
with respect to the Gromov-Hausdorff topology to a flat manifold.

In Theorem \ref{th:1aa}, we consider a sequence of closed Riemannian $4$-manifolds, whose $L^2$-norm of the 
Riemannian curvature tensor is uniformly bounded from above, and whose $L^2$-norm of the 
traceless Ricci-tensor tends to zero. Here, under the assumption of a uniform non-collapsing bound, which is very close to the euclidean situation, and a uniform diameter bound, we
show that there exists a subsequence which converges in the Gromov-Hausdorff sense to an Einstein manifold.

In order to prove Theorem \ref{th:1} and Theorem \ref{th:1aa}, we use a smoothing technique, which is called $L^2$-curvature flow or $L^2$-flow, introduced by Jeffrey Streets in the series of works 
\cite{streets2008gradient}, \cite{streets2012gradientsmall}, \cite{streets2012gradientround}, \cite{streets2013collapsing}, \cite{streets2013long} and
\cite{streets2016concentration}. In particular, we use his  ``tubular averaging technique'', which he has introduced in \cite[Section 3]{streets2016concentration}, in order to prove distance estimates of the 
$L^2$-curvature flow which only depend on significant geometric bounds. This is the content of Theorem \ref{th:2a}.
\end{abstract}

\maketitle 
\tableofcontents
\section{Introduction and statement of results}

In order to approach minimization problems in Riemannian geometry, it is often useful to know if a minimizing sequence of smooth Riemannian manifolds
contains a subsequence that converges with respect to an appropriate topology to a sufficiently smooth space. Here, in general, the minimization problem refers
to a certain geometric functional, for instance the area functional, the total scalar curvature functional, the Willmore functional or the $L^p$-norm of
a specific curvature tensor on a Riemannian manifold, to name just a few. Latter functionals are the main interest in this work. That means that we consider sequences of
Riemannian manifolds that have a uniform $L^p$-bound on the full curvature tensor, the Ricci tensor and the traceless Ricci tensor respectively.

Naturally, the situation is more transparent, if we have more precise information about the $L^p$-boundedness of curvature tensors of the underlying Riemannian manifolds, that is, that we have
a uniform $L^p$-bound, where $p\in [1,\infty]$ is large. In particular, a uniform $L^{\infty}$-bound should give the most detailled information about geometric quantities.

One of the basic results in this context is stated in \cite[Theorem 2.2, p. 464-466]{anderson1989ricci}. Here, for instance, one assumes a uniform $L^{\infty}$-bound
on the full Riemannian curvature tensor, a uniform lower bound on the injectivity radius and a uniform two sided bound
on the volume, to show the existence of a subsequence that converges with respect to the $C^{0,\alpha}$-topology to a 
Riemannian manifold of regularity $C^{1,\alpha}$. The proof uses the fact, that it is possible to find uniform 
coverings of the underlying manifolds with harmonic charts, which follows from \cite{jost1982geometrische}.

In \cite{yang1992p}, Deane Yang has considered sequences of Riemannian manifolds satisfying a suitable uniform $L^p$-bound on their full Riemannian curvature tensors, where $p>\frac{n}{2}$, and 
a uniform bound on the Sobolev constant. In order to show compactness and diffeomorphism finiteness results, he examines Hamilton's Ricci flow (cf. \cite{hamilton1982three}, \cite{chow2006hamilton} and \cite{topping2006lectures}) and he shows curvature decay estimates and existence time estimates that only depend on the significant geometric bounds. 

In \cite{yang1992convergence} and \cite{yang1992convergenceII}, Deane Yang has approached a slightly more general problem. Here, he has considered sequences of Riemannian
$n$-manifolds, $n\geq 3$, having a uniform $L^{\frac{n}{2}}$-bound on their full Riemannian curvature tensors and a suitable uniform $L^p$-bound on their Ricci tensors instead of a
uniform $L^p$-bound on their full Riemannian curvature tensors, where $p>\frac{n}{2}$. Due to the scale invariance of the bound on the Riemannian curvature tensors - we name such bound a 
``critical curvature bound''- the situation
becomes much more difficult, than in the ``supercritical'' case, that is, when $p$ is bigger than $\frac{n}{2}$.
In particular, in general, it is doubtful whether the global Ricci flow is applicable in this situation.

In \cite{yang1992convergence}, the author has introduced the idea of a
``local Ricci flow'' which is, by definition, equal to the Ricci flow weighted with a truncation function that is compactly contained in a local region of a manifold. 
The author shows that on regions, where the local $L^{\frac{n}{2}}$-norm of the full Riemannian curvature tensor is sufficiently small, the local Ricci flow
satisfies curvature decay estimates and existence time estimates that only depend on significant local geometric bounds. So, on these ``good'', regions one may apply
\cite[Theorem 2.2, pp. 464-466]{anderson1989ricci} to a slightly mollified metric, to obtain local compactness with respect to the $C^{0,\alpha}$-topology.
Since the number of local regions having too large $L^{\frac{n}{2}}$-norm of the full Riemannian curvature tensor is uniformly bounded, the author is able to show
that each sequence of closed Riemannian manifolds, satisfying a uniform diameter bound, a uniform non-collapsing bound on the volume of small balls, 
a uniform bound on the $L^{\frac{n}{2}}$-norm of the full Riemannian curvature tensor and a sufficiently small uniform bound on the $L^{p}$-norm of the Ricci curvature tensor,
where $p>\frac{n}{2}$, contains a subsequence that converges in the Gromov-Hausdorff sense to a metric space, which is, outside of a finite set of points, an open $C^1$-manifold with a Riemannian metric of regularity $C^0$.
 
In \cite{yang1992convergenceII}, the author has used the local Ricci flow to find a suitable harmonic chart around each point in whose neighborhood the 
local $L^{\frac{n}{2}}$-norm of the full Riemannian curvature tensor and the local $L^{p}$-norm of the full Riemannian curvature tensor,
where $p>\frac{n}{2}$, is not too large. Using these estimates, the author is able to improve the statements about the convergence behavior in the convergence results in \cite{yang1992convergence}
on regions having a sufficiently small curvature concentration.

It seems so, that the reliability of the Ricci flow in \cite{yang1992p}, and the local Ricci flow in \cite{yang1992convergence} and \cite{yang1992convergenceII} is based
on the appearance of the supercritical curvature bounds. For instance, in order to develop
the parabolic Moser iteration in \cite{yang1992p} and \cite{yang1992convergence} one uses a well-controlled behavior of the Sobolev constant. As shown in \cite[7, pp. 85-89]{yang1992convergence} this
behavior occurs, if one assumes suitable supercritical bounds on the Ricci curvature. The examples in \cite[Section 9, pp. 690-694]{aubry2007finiteness} show that the critical case is completely different.

Another important issue is the absence of important comparison geometry results under critical curvature bounds.
In order to understand the rough structure of Riemannian manifolds, satisfying a fixed lower bound on the Ricci tensor, one uses the well-known ``Bishop-Gromov volume comparison theorem'' \index{Bishop-Gromov volume comparison theorem}\, 
(cf. \cite[9.1.2., pp. 268-270]{petersen2006riemannian}) which allows a one-directed volume comparison of balls in Riemannian manifolds
satisfying a fixed lower Ricci curvature bound with the volume of balls in a such called ``space form'', (cf. \cite[p. 206]{lee1997riemannian})\index{space form}, which is 
a complete, connected Riemannian manifold with constant sectional curvature. Later, in \cite{petersen1997relative}, Peter Petersen and Guofang Wei have shown that it is possible to 
generalize this result to the situation, in that an $L^p$-integral of some negative part of the Ricci tensor is sufficiently small. Here the authors assume that $p$ is 
bigger that $\frac{n}{2}$. 

It seems that the treatment of Riemannian manifolds with pure critical curvature bounds needs to be based on methods that are different from the approaches we have just mentioned.
Instead of considering the Ricci flow,
which is closely related to the gradient flow of the Einstein-Hilbert functional (cf. \cite[Chapter 2, Section 4, pp. 104-105]{chow2006hamilton}), 
one could try to deform a Riemannian manifold of dimension $4$ into the direction
of the negative gradient of the $L^2$-integral of the full curvature tensor, in order to analyze slightly deformed approximations of the initial metric, having a smaller curvature energy
concentration. This evolution equation was examined by 
Jeffrey Streets in \cite{streets2008gradient}, \cite{streets2012gradientsmall}, \cite{streets2012gradientround}, \cite{streets2013collapsing}, \cite{streets2013long},
\cite{streets2016concentration}. In this series of works, J. Streets has proved a plenty of properties of this geometric flow and he also shows a couple of applications.

Using J. Streets technique, we show 
compactness results for Riemannian $4$-manifolds, that only assume a uniform diameter bound, a uniform non-collapsing bound on the volume of sufficiently small balls
and critical curvature bounds.

In the first theorem, we consider a sequence of Riemannian $4$-manifolds having almost vanishing Riemannian curvature tensor 
in some rough sense and we show that a subsequence converges with respect to the Gromov-Hausdorff topology to a flat Riemannian manifold:
 \begin{theorem}\label{th:1}
Given $D,d_0>0$,  $\delta\in(0,1)$ and let $(M_i,g_i)_{i\in\mathbb{N}}$ be a sequence of closed Riemannian 4-manifolds, satisfying the following assumptions:
\begin{alignat}{2}
\notag d_0\leq diam_{g_i}(M_i) &\leq D &&\forall i\in\mathbb{N} \\
\notag  Vol_{g_i}(B_{g_i}(x,r)) &\geq \delta\omega_4 r^4\hspace{1cm} && \forall i\in\mathbb{N},\, x\in M_i,\, \forall r\in [0,1] \\
\label{eq:2.1+}
\left\Vert Rm_{g_i} \right\Vert_{L^2(M_i,g_i)}&\leq \frac{1}{i}\hspace{1cm} &&\forall i\in\mathbb{N}
\end{alignat}
then, there exists a subsequence $(M_{i_j},d_{g_{i_j}})_{j\in\mathbb{N}}$ that converges in the Gromov-Haus\-dorff sense to a smooth flat manifold $(M,g)$.
\end{theorem}
Throughout, a closed Riemannian\index{closed Riemannian manifold} is defined to be a smooth, compact and connected oriented Riemannian manifold without boundary.

In the second theorem, we consider a sequence of Riemannian $4$-manifolds with uniformly bounded curvature energy and almost vanishing traceless Ricci tensor in 
some rough sense. Under these assumptions, we show
 that a subsequence converges with respect to the Gromov-Hausdorff topology to an Einstein manifold, provided that the volume of small balls behaves almost euclidean:
\begin{theorem}\label{th:1aa}
Given $D,d_0,\Lambda>0$, there exists a universal constant $\delta\in(0,1)$ close to $1$ so that if 
$(M_i,g_i)_{i\in\mathbb{N}}$ is a sequence of closed Riemannian 4-manifolds
 satisfying the following assumptions:
\begin{alignat*}{2}
d_0\leq  diam_{g_i}(M_i) &\leq D &&\forall i\in\mathbb{N} \\
\Vert Rm_{g_i} \Vert_{L^2(M_i,g_i)}&\leq \Lambda\hspace{1cm} &&\forall i\in\mathbb{N}\\
\Vert \mathring{Rc}_{g_i} \Vert_{L^2(M_i,g_i)}&\leq \frac{1}{i}\hspace{1cm} &&\forall i\in\mathbb{N}\\
 Vol_{g_i}(B_{g_i}(x,r)) &\geq \delta\omega_4 r^4\hspace{1cm} && \forall i\in\mathbb{N},\, x\in M_i,\, r\in [0,1]
\end{alignat*}
then there exists a subsequence $(M_{i_j},d_{g_{i_j}})_{j\in\mathbb{N}}$ that converges in the Gromov-Haus\-dorff sense to a smooth Einstein manifold $(M,g)$.
\end{theorem}

As mentioned above, it is our aim to show these results, using the negative gradient flow of the following functional:
\begin{equation}
\label{eq:2.4}
\mathcal{F}(g):=\int_{M}{|Rm_{g}|_g^2\, dV_g}
\end{equation}

That is, on a fixed sequence element $(M^4,g_0)$, we want to evolve the initial metric in the following manner:

\begin{equation}
\label{eq:2.5}
\begin{cases}
\frac{\partial}{\partial t}g&=-\text{grad }\mathcal{F}=-2\delta d Rc_g +2\check{R}_{g} -\frac{1}{2} |Rm_{g}|^2_{g} g \\
g(0)&=g_0
\end{cases}
\end{equation}
where $\check{R}_{ij}:=R_i^{pqr}R_{jpqr}$ in local coordinates and the gradient formula, which appears in \eqref{eq:2.5}  can be found in \cite[Chapter 4, 4.70 Proposition, p. 134]{besse2007einstein}. Here,
$d$ denotes the exterior derivative \index{exterior derivative} acting on the Ricci tensor 
and $\delta$ denotes the adjoint of $d$.\index{exterior derivative>adjoint of the $\sim$} The gradient of a differentiable Riemannian functional is defined in \cite[Chapter 4, 4.10 Definition, p. 119]{besse2007einstein}.

In \cite[Theorem 3.1, p. 252]{streets2008gradient} J. Streets has proved short time existence of the flow given by \eqref{eq:2.5} on closed Riemannian manifolds. The author has also proved the uniqueness of the flow (cf. \cite[Theorem 3.1, p. 252]{streets2008gradient}). In this regard, the
expression ``the'', $L^2$-flow makes sense. In \cite[Theorem 1.8, p. 260]{streets2016concentration} J. Streets has proved, that
under certain assumptions, the flow given by \eqref{eq:2.5} has a solution on a controlled time interval and the solution satisfies certain
curvature decay and injectivity radius growth estimates.

In Section \ref{sec:2}, we use J. Streets ideas, in order to show that, under certain assumptions, the distance between two points does not change too much along the flow. This allows us
to bring the convergence behavior of a slightly mollified manifold back to the initial sequence. That means we will prove the following theorem:
\pagebreak
\begin{theorem}\label{th:2a}
Let $(M^4,g_0)$ be a closed Riemannian 4-manifold. Suppose that $(M,g(t))_{t\in[0,1]}$ is a solution to \eqref{eq:2.5} satisfying the following assumptions:
\begin{alignat}{2}
\label{eq:2.6-} \int_{M}{|Rm_{g_0}|_{g_0}^2\, dV_{g_0}}&\leq \Lambda && \\
\label{eq:2.6} \Vert Rm_{g(t)}\Vert_{L^{\infty}(M,g(t))}&\leq K t^{-\frac{1}{2}} &&\hspace{0.25cm}\forall t\in (0,1]\\
\label{eq:2.6a}
inj_{g(t)}(M)&\geq \iota t^{\frac{1}{4}} &&\hspace{0.25cm}\forall t\in [0,1]\\
\label{eq:2.6b}
diam_{g(t)}(M)&\leq 2(1+D) &&\hspace{0.25cm}\forall t\in [0,1]
\end{alignat}
Then we have the following estimate:
\begin{equation}\label{disex}
|d(x,y,t_2)-d(x,y,t_1)|\leq C(K,\iota,D)\Lambda^{\frac{1}{2}}\left(t_2^{\frac{1}{8}}-t_1^{\frac{1}{8}}\right)^{\frac{1}{2}} +C(K,\iota,D) \left( t_2^{\frac{1}{24}}-t_1^{\frac{1}{24}} \right)
\end{equation}
for all $t_1,t_2\in [0,1]$ where $t_1<t_2$.
\end{theorem}

These estimates allow one to prove Theorem \ref{th:1} and Theorem \ref{th:1aa} which are the main goals of Section \ref{sec:3} and Section \ref{prth12}. Here, in Section \ref{sec:3},
we may refer to the estimates 
in \cite[1.3, Theorem 1.8, p. 260]{streets2016concentration}. In Section \ref{prth12}, we write down an existence result which allows to apply Theorem \ref{th:2a} to the elements of the
sequence occurring in Theorem \ref{th:1aa}.

This work is a part of the author's doctoral thesis (\cite{ZergaengeDiss}), written under the supervision of Miles Simon at the Otto-von-Guericke-Universit\"at Magdeburg.


\section{\texorpdfstring{Distance control under the $L^2$-flow in $4$ dimensions}{Distance control under the L2-flow in 4 dimensions}}\label{sec:2}

In this section we prove Theorem \ref{th:2a}. 
In order to prove this Theorem we use the ``tubular averaging technique''  from \cite[Section 3, pp. 269-282]{streets2016concentration}. 
The method is derived from \cite[Section 3]{streets2016concentration}. 
In Subsection \ref{Backward estimates}, we apply the ``tubular averaging technique'' to the time-reversed flow. 
For the sake of understanding, we give detailed explanations of the steps in the proof, even if the argumentation is based on the content of \cite[Section 3]{streets2016concentration}.
In order to get a very rough feeling for J. Streets ``tubular averaging technique'' we recommend to read the first paragraph of \cite[p. 270]{streets2016concentration}

The proof of Theorem \ref{th:2a} is divided in two principal parts: 

In the first part of this section
we show that, along the flow, the distance between two points in manifold $M$ does not increase too much, i.e.: we derive the estimates of the shape
 $d(x,y,t)<d(x,y,0)+\epsilon$ for small $t(\epsilon)>0$. We say that this kind of an estimate is a ``forward estimate''.

The second part in this section is concerned with the opposite direction, i.e: we show that, along the flow, the distance between two points does not decay too much, 
which means that we have
$d(x,y,t)>d(x,y,0)-\epsilon$ for $t(\epsilon)>0$ sufficiently small. 

We point out that the estimate of the length change of a vector $v\in TM$ along a geometric flow usually requires an integration of the metric change $|g'(t)|_{g(t)}$ from $0$ to a later time point
 $T$ (cf. \eqref{ap:1a}). With a view to \eqref{eq:2.5}
and \eqref{eq:2.6} we note that, on the first view, this would require and integration of the function $t^{-1}$ from $0$ to $T$ which is not possible. 

In order to overcome this difficulty, we follow the ideas
in \cite[Section 3]{streets2016concentration}, 
i.e. we introduce some kind of connecting curves which have almost the properties of geodesics. Then we construct an appropriate tube around each of these connecting curves so that the
 integral $ \int_{\gamma}{ \left|\text{grad }\mathcal{F}\right|\, d\sigma  }$, which occurs in the estimate of $\left|\frac{d}{ dt} L(\gamma,t)\right|$ (cf. \eqref{ap:1}), can be estimated from above against a well-controlled average
integral along the tube plus an error integral which behaves also well with respect to $t$. 
We point out that we do not widen J. Streets ideas in \cite[Section 3]{streets2016concentration} by fundamental facts, we merely write down detailed information which allow to understand the distance changing behavior of J. Streets $L^2$-flow in a more detailed way.

\subsection{Tubular neighborhoods}
We quote the following definition  from {\cite[Definition 3.3., pp. 271-272]{streets2016concentration}}
\begin{defi}\label{tubedefi}
Let $(M^n,g)$ be a smooth Riemannian manifold without boundary, and let $\gamma: [a,b]\longrightarrow M$ be an smooth curve. Given $r>0$, and $s\in [a,b]$ then we define
\begin{align*}
D(\gamma(s),r)&:= \exp_{\gamma(s)}\left(B(0,r) \cap \langle \dot{\gamma}(s)\rangle^{\perp}  \right)
\intertext{and}
D(\gamma,r)&:= \bigcup_{s\in [a,b]}{D(\gamma(s),r)}
\end{align*}
We say  ``$D(\gamma,r)$ is foliated by $\left( D(\gamma(s),r) \right)_{s\in [a,b]}$'' if  
\begin{equation*}
 D(\gamma(s_1),{r}) \cap  D(\gamma(s_2),{r})= \emptyset
\end{equation*}
for all $a\leq s_1 < s_2 \leq b$.
\end{defi}
The following definition is based on {\cite[Definition 2.2., p. 267]{streets2016concentration}}.
\begin{defi}\label{fk}
Let $(M^n,g)$ be a closed Riemannian manifold, $k\in\mathbb{N}$ and $x\in M$, then we define 
\begin{align*}
f_k(x,g):=\sum_{j=0}^k{{| ^g\nabla^j Rm_{g}|_g^{\frac{2}{2+j}}}}(x)
\intertext{and}
f_k(M,g):=\sup_{x\in M}{f_k(x,g)}
\end{align*}
\end{defi}
At this point we refer to the scaling behavior of $f_k(x,g)$ which is outlined in Lemma \ref{fkscaling}.

The following result is a slight modification of {\cite[Lemma 3.4., pp. 272-274]{streets2016concentration}}. To be more precise: in this result we
allow the considered curve to have a parametrization close
to unit-speed, and not alone unit-speed.
\begin{lem}\label{tube}
Given $n,D,K,\iota>0$ there exists a constant $\beta(n,D,K,\iota)>0$ and a constant $\mu(n)>0$ so that if $(M^n,g)$ is a 
complete Riemannian manifold satisfying
\begin{align*}
diam_g(M)  &\leq D \\
f_3(M^n,g) &\leq K \\
inj_g(M)   &\geq \iota
\end{align*}
and $\gamma: [0,L]\rightarrow M$ is an injective smooth curve satisfying
\begin{align}
\label{almlg}
L(\gamma)&\leq d(\gamma(0),\gamma(L))+\beta \\
\label{almgeod}
| \nabla_{\dot{\gamma}}\dot{\gamma}|&\leq \beta \\
\label{almunit}
\frac{1}{1+\beta}\leq |\dot{\gamma}| &\leq 1+\beta
\end{align}
then $D(\gamma,R)$ is foliated by $\left( D(\gamma(s),R) \right)_{s\in [0,L]}$ for $R:=\mu \min\left\{\iota,K^{-\frac{1}{2}} \right\}$. Furthermore, if 
\begin{equation*}
\pi : D(\gamma,R)\longrightarrow \gamma([0,L]) 
\end{equation*}
is the projection map \index{projection map} sending a point $q\in D(p,R)$, where $p\in \gamma([0,L])$, to $p$, which is well-defined by the foliation property, then
\begin{equation}\label{dpi}
|d \pi |\leq 2\text{ on }D(\gamma,R)
\end{equation}
Here $d \pi$ denotes the differential and $|d \pi |$ denotes the operator norm of the differential of the projection map. 
\end{lem}
\begin{proof} 
Above all, we want to point out, that, due to the injectivity of the curve, we can construct a tubular neighborhood around $\gamma([0,L])$. This is a consequence of \cite[26. Proposition, p. 200]{o1983semi}.
 But the size of this neighborhood is not controlled at first. Via radial projection we can ensure
that the velocity field of the curve is extendible in the sense of \cite[p. 56]{lee1997riemannian}.
We follow the ideas of the proof of \cite[Lemma 3.4, pp. 272-274]{streets2016concentration} with some modifications.

Firstly, we describe how $\mu(n)>0$ needs to be chosen in order to ensure that the curve has a suitable foliation which can be used to define the projection map.

Secondly, we show that the desired smallness condition of the derivative of the projection map is valid, i.e.: we show \eqref{dpi}. Here we allow $\mu(n)>0$ to become smaller.

Let
\begin{equation}
\mu(n):=\min\left\{ \widehat{\mu}(n),  \frac{1}{20} ,\frac{1}{64 C_1(n)C_2(n)}   \right\}
\end{equation}
where $\widehat{\mu}(n)>0$ and $C_1(n)>0$ are taken from \cite[Lemma 2.9, p. 268]{streets2016concentration} and $C_2(n)>0$ will be made explicit below. Let 
\begin{equation*}
R:=\mu \min\left\{\iota,K^{-\frac{1}{2}} \right\}
\end{equation*}
 Suppose there exists a point $p\in D(\gamma(s_0),R)\cap D(\gamma(s_1),R)$
 where $s_0,s_1\in [0,L]$, $s_0<s_1$ and  $s_1-s_0\leq 10 R$ at first. By definition, there exists a normal chart of radius $20R$ around $p$  (cf. \cite[pp. 76-81]{lee1997riemannian}).
In this chart we have the following estimate
\begin{equation}\label{eq:mu}
\sup_{B_g(p,20R)}{\mu K^{-\frac{1}{2}}|\Gamma|}\leq \frac{1}{64 C_2(n)}
\end{equation}
Choosing $\beta\in (0,1)$ small enough compared to $R$ we ensure that  $\gamma([s_0,s_1])$ lies in this chart. From \cite[Theorem 6.8., pp. 102-103]{lee1997riemannian} we obtain
\begin{equation*}
\left.\left\langle \frac{\partial}{\partial r}, \dot{\gamma}\right\rangle\right|_{\gamma(s_0)}=0
\end{equation*}
where 
\begin{equation}\label{rad}
\left.\frac{\partial}{\partial r}\right|_{\gamma(s)}:=\frac{\gamma^i(s)}{r(\gamma(s))} \left.\partial_i\right|_{\gamma(s)}
\end{equation}
and $\partial_1,...,\partial_n$ denote the coordinate vector fields and $\gamma^1,...,\gamma^n$ denote the coordinates
of $\gamma$ in this normal chart and 
\begin{equation*}
r(\gamma(s)):=\sqrt{\sum_{i=1}^n{(\gamma^i(s))^2}}
\end{equation*}
(cf. \cite[Lemma 5.10, (5.10), p. 77]{lee1997riemannian}). We show that it is possible to take $\beta(n,K,\iota)>0$ small enough to ensure that
\begin{equation*}
\left.\left\langle \frac{\partial}{\partial r}, \dot{\gamma}\right\rangle\right|_{\gamma(s)}\neq 0\hspace{1cm}\forall s\in (s_0,s_1]
\end{equation*}
This would be a contradiction to the fact that \cite[Theorem 6.8., pp. 102-103]{lee1997riemannian} also implies 
\begin{equation}\label{rad2}
\left.\left\langle \frac{\partial}{\partial r}, \dot{\gamma}\right\rangle\right|_{\gamma(s_1)}=0
\end{equation}
From \cite[Lemma 5.2 (c), p. 67]{lee1997riemannian} we infer on $[s_0,s_1]$
\begin{align}\label{partcov}
\begin{split}
\frac{\partial}{\partial s}\left.\left\langle \frac{\partial}{\partial r}, \dot{\gamma}\right\rangle\right|_{\gamma(s)}=&\left.\left\langle D_s{\frac{\partial}{\partial r}} , \dot{\gamma} \right\rangle\right|_{\gamma(s)}+\left.\left\langle {\frac{\partial}{\partial r}} , D_s \dot{\gamma} \right\rangle\right|_{\gamma(s)}\\
\geq&\left.\left\langle D_s{\frac{\partial}{\partial r}} , \dot{\gamma}  \right\rangle\right|_{\gamma(s)}-\left|\left.\left\langle {\frac{\partial}{\partial r}} , D_s \dot{\gamma}  \right\rangle\right|_{\gamma(s)}\right|\\
\geq & \left.\left\langle D_s{\frac{\partial}{\partial r}} , \dot{\gamma}  \right\rangle\right|_{\gamma(s)}-\left |  \left. \frac{\partial}{\partial r} \right|_{\gamma(s)} \right|
  \left|\nabla_{\dot{\gamma}(s)} \dot{\gamma}(s) \right|_g\\
\stackrel{\eqref{almgeod}}{\geq}& \left.\left\langle D_s{\frac{\partial}{\partial r}} , \dot{\gamma}  \right\rangle\right|_{\gamma(s)}- \beta
\end{split}
\end{align}
Using \eqref{rad} together with \cite[Lemma 4.9 (b), p. 57]{lee1997riemannian} and \cite[p. 56 (4.9)]{lee1997riemannian} we calculate
\begin{equation*}
D_s \frac{\partial}{\partial r}=\frac{ \dot{\gamma}^i \cdot r -  \gamma^i \left\langle \dot{\gamma}, \frac{\partial }{\partial r}\right\rangle}{r^2}  \partial_i+\frac{\gamma^i}{r} D_s \partial_i
=\frac{\dot{\gamma}^i}{r}\partial_i - \frac{   \gamma^i \left\langle \dot{\gamma}, \frac{\partial }{\partial r}\right\rangle}{r^2}  \partial_i+\frac{\gamma^i}{r} D_s \partial_i
\end{equation*}
This implies
\begin{align*}
\left\langle D_s \frac{\partial}{\partial r} , \dot{\gamma} \right\rangle=&\frac{1}{r} |\dot{\gamma}|^2-\frac{1}{r^2} \left\langle \dot{\gamma}, \frac{\partial }{\partial r}\right\rangle \left\langle 
 \gamma^i \partial_i  ,\dot{\gamma}  \right\rangle +\frac{\gamma^i}{r} \left\langle  D_s \partial_i , \dot{\gamma} \right\rangle\\
\stackrel{\eqref{rad}}{\geq} & \frac{1}{r} |\dot{\gamma}|^2-\frac{1}{r^2} \left|\left\langle \dot{\gamma}, \frac{\partial }{\partial r}\right\rangle\right| \left| \left\langle  r \frac{\partial}{\partial r}  
,\dot{\gamma}  \right\rangle \right|-   C_2(n)|\Gamma| | \dot{\gamma} |^2\\
\stackrel{\hphantom{\eqref{almunit}}}{\geq} & \frac{1}{r} |\dot{\gamma}|^2-\frac{1}{r} |\dot{\gamma}| \left| \left\langle  \frac{\partial}{\partial r}  ,\dot{\gamma}  \right\rangle \right|-C_2(n)
   |\Gamma| | \dot{\gamma} |^2\\
\stackrel{\eqref{almunit}}{\geq} & \frac{1}{4r} -\frac{2}{r}  \left| \left\langle  \frac{\partial}{\partial r}  ,\dot{\gamma}  \right\rangle \right|- 4 C_2(n) |\Gamma|\\
 =&\frac{1-4 C_2(n)\left| \left\langle  
\frac{\partial}{\partial r}  ,\dot{\gamma}  \right\rangle \right|- 16 C_2(n) r |\Gamma|}{4r}
\end{align*}
Here, in order to obtain the first estimate, we refer to Definition \ref{Gamma} and the fact that
\begin{equation*}
 \frac{1}{r(\gamma(t))}\sum_{i=1}^n{|\gamma^i(t)|}
\leq \widehat{C}(n)
\end{equation*}
Hence, \eqref{partcov} implies
\begin{align}
\begin{split}\label{DIneq}
\frac{\partial}{\partial s}\left\langle \frac{\partial}{\partial r}, \dot{\gamma}\right\rangle&\geq \frac{1-8\left| \left\langle  \frac{\partial}{\partial r}  ,\dot{\gamma}  \right\rangle \right|- 16 C_2(n) r|\Gamma|-4\beta r}{4r}\\
&\geq  \frac{1-8\left| \left\langle  \frac{\partial}{\partial r}  ,\dot{\gamma}  \right\rangle \right|- 16 C_2(n) \mu K^{-\frac{1}{2}}|\Gamma|-4\mu K^{-\frac{1}{2}}\beta }{4r}\\
&\stackrel{\eqref{eq:mu}}{\geq}  \frac{1-8\left| \left\langle  \frac{\partial}{\partial r}  ,\dot{\gamma}  \right\rangle \right|- \frac{1}{4}-4\mu K^{-\frac{1}{2}}\beta }{4r}\\
&\stackrel{\hphantom{\eqref{eq:mu}}}{\geq}  \frac{1-8\left| \left\langle  \frac{\partial}{\partial r}  ,\dot{\gamma}  \right\rangle \right|- \frac{1}{4}-\frac{1}{4} }{4r}\\
&\stackrel{\hphantom{\eqref{eq:mu}}}{=}  \frac{\frac{1}{2}-8\left| \left\langle  \frac{\partial}{\partial r}  ,\dot{\gamma}  \right\rangle \right|}{4r}=\frac{1}{8r} \left[ 1-16\left| \left\langle  \frac{\partial}{\partial r}  ,\dot{\gamma}  \right\rangle \right|\right]
\end{split}
\end{align}
We show that this differential inequality implies the desired contradiction. Let $w: [s_0,s_1]\longrightarrow \mathbb{R}$, $w(s):=\left.\left\langle \frac{\partial}{\partial r}, \dot{\gamma}\right\rangle\right|_{\gamma(s)}$. Then \eqref{DIneq} is equivalent to
\begin{equation*}
w'\geq \frac{1}{8r}(1-16 |w|)
\end{equation*}
on $[s_0,s_1]$. Since $w(s_0)=0$, there exists $\delta>0$ such that $w'>0$ on $[s_0,s_0+\delta]$. This implies  $w>0$ on $(s_0,s_0+\delta]$. 
We show that we have $w>0$ on $(s_0,s_1]$, which contradicts \eqref{rad2}.
Assumed
\begin{equation*}
\widehat{s}:=\sup\left\{s\in (s_0,s_1]|\, \left. w \right|_{(s_0,s]}>0 \right\}<s_1
\end{equation*}
which implies 
\begin{equation}\label{contr1}
w(\widehat{s})=0
\end{equation}
Then \eqref{DIneq} is equivalent to
\begin{equation*}
w'\geq \frac{1}{8r}(1-16 w)
\end{equation*}
on $[s_0,\widehat{s}]$. The function $z: [s_0,s_1]\longrightarrow \mathbb{R}$, $z(s):=\frac{1}{16}(1-e^{-\frac{2(s-s_0)}{r}})$ satisfies $z'= \frac{1}{8r}(1-16 z)$ on  $[s_0,s_1]$ and $z(s_0)=0$.
Thus we have
\begin{equation}\label{DIneq2}
\begin{cases}
(w-z)'\geq -\frac{2}{r}(w-z) & \text{on }[s_0,\widehat{s}]\\
(w-z)(s_0)=0
\end{cases}
\end{equation}
and we define a new function $\zeta : [s_0,s_1]\longrightarrow \mathbb{R}$ as follows $\zeta(s):=e^{\frac{2}{r}s}(w(s)-z(s))$. Then \eqref{DIneq2} implies
\begin{align*}
\zeta'(s)= &\frac{2}{r} e^{\frac{2}{r}s}(w(s)-z(s))+e^{\frac{2}{r}s}(w'(s)-z'(s))\\
\geq & \frac{2}{r} e^{\frac{2}{r}s}(w(s)-z(s))-\frac{2}{r}e^{\frac{2}{r}s}(w(s)-z(s))=0
\end{align*}
Hence
\begin{align*}
e^{\frac{2}{r}\widehat{s}}(w(\widehat{s})-z(\widehat{s}))=\zeta(\widehat{s})=\int_{s_0}^{\widehat{s}}{\zeta'(\tau)\ d\tau}\geq 0
\end{align*}
from this we obtain
\begin{equation*}
w(\widehat{s})\geq z(\widehat{s})=\frac{1}{16}(1-e^{-\frac{2(\widehat{s}-s_0)}{r}})>0
\end{equation*}
which contradicts \eqref{contr1}. Consequently, we have $w\geq 0$ on $[s_0,s_1]$. The same argumentation as above, adapted to the interval $[s_0,s_1]$, implies $w(s_1)>0$ in contradiction to \eqref{rad2}. This proves that two discs $D(\gamma(s_0),R)$ and $D(\gamma(s_1),R)$ cannot intersect, when $|s_1-s_0|\leq 10 R$.

Now, we show that two discs $D(\gamma(s_0),R)$ and  $D(\gamma(s_1),R)$ cannot intersect if we assume  $s_0,s_1\in [0,L]$, $s_0<s_1$, to be far away from each other, which means that $s_1-s_0>10R$ holds. 

We suppose that there exists a point $p\in D(\gamma(s_0),R)\cap D(\gamma(s_1),R)$. As in \cite[p. 273]{streets2016concentration} we
 construct a curve $\alpha$ in the following manner: $\alpha$ follows $\gamma$ from $\gamma(0)$ to $\gamma(s_0)$, next $\alpha$ 
connects $\gamma(s_0)$ and $p$ by a minimizing geodesic, 
then $\alpha$ connects  $p$ and $\gamma(s_1)$ also by a minimizing geodesic, and finally $\alpha$ follows $\gamma$ again from $\gamma(s_1)$ to $\gamma(L)$. We infer the following estimate:
\begin{align*}
d_g(\gamma(0),\gamma(L))\leq  L(\alpha)\leq &\int_0^{s_0}{ |\dot{\gamma} |\, ds} +R+R+ \int_{s_1}^{L}{|\dot{\gamma} |\, ds}\\
\stackrel{\eqref{almunit}}{\leq}& (1+\beta)s_0+ 2R+(1+\beta)(L-s_1)\\
\stackrel{\hphantom{\eqref{almunit}}}{=}& (1+\beta)L+ 2R-(1+\beta)(s_1-s_0)\\
\stackrel{\hphantom{\eqref{almunit}}}{=}& (1+\beta)\int_0^{L}{ \frac{|\dot{\gamma} |}{|\dot{\gamma} |}\, ds}+ 2R-(1+\beta)(s_1-s_0)\\
\stackrel{\eqref{almunit}}{\leq}& (1+\beta)^2\int_0^{L}{ |\dot{\gamma} |\, ds}+ 2R-(1+\beta)10R\\
\stackrel{\hphantom{\eqref{almunit}}}{\leq}& (1+\beta)^2L(\gamma)-8R \\
\stackrel{\eqref{almlg}}{\leq}& (1+\beta)^2(d_g(\gamma(0),\gamma(L))+\beta)-8R \\
\stackrel{\hphantom{\eqref{almlg}}}{\leq}& (1+2\beta+\beta^2)(d_g(\gamma(0),\gamma(L))+4\beta-8R \\
\stackrel{\hphantom{\eqref{almlg}}}{\leq}& (d_g(\gamma(0),\gamma(L))+3\beta D+4\beta-8R
\end{align*}
and consequently:
\begin{equation*}
0 \leq (3D+4)\beta-8R
\end{equation*}
which yields a contradiction when $\beta(n,D,K,\iota)>0$ is chosen small enough. Hence, two discs $D(\gamma(s_0),R)$ and  $D(\gamma(s_1),R)$ cannot intersect, provided they are not identical. Thus, 
$D(\gamma,R)$ is foliated by $\left( D(\gamma(s),R) \right)_{s\in [0,L]}$.


 It remains to show the estimate \eqref{dpi}.
We mentioned at the beginning of the proof, that now,  we allow $\mu$ to become smaller.

 As in the proof of \cite[Lemma 3.4.]{streets2016concentration} we suppose the assertion would be not true, i.e.
 there exists a sequence of constants $(\mu_i)_{i\in\mathbb{N}}$, where $\lim_{i\to\infty}{\mu_i}=0$,
 and a sequence of closed Riemannian manifolds
$(M_i^n,g_i)_{i\in\mathbb{N}}$ satisfying 
\begin{align*}
f_3(M_i,g_i)&\leq K_i\hspace{0.5cm}\text{and}\\
inj_{g_i}(M_i)&\geq \iota_i
\end{align*}
for all $i\in\mathbb{N}$, and curves $\gamma_i: [0,L_i]\longrightarrow M_i$ satisfying
\begin{align}
\begin{split}
L(\gamma_i)&\leq d(\gamma_i(0),\gamma_i(L_i))+\beta_i,\\
| \nabla_{\dot{\gamma_i}}\dot{\gamma_i}|&\leq \beta_i \hspace{0.5cm}\text{and}\\
\frac{1}{1+\beta_i} \leq |\dot{\gamma_i}|&\leq 1+\beta_i 
\end{split}
\end{align}
for all  $i\in\mathbb{N}$, where $(\beta_i)_{i\in\mathbb{N}}\subseteq (0,1]$, so that for each $i\in\mathbb{N}$ the tube $D(\gamma_i,R_i)$
is foliated by $\left( D(\gamma_i(s),R_i)\right)_{s\in [0,L_i]}$, where $R_i:= \mu_i\min\left\{\iota_i,K_i^{-\frac{1}{2}} \right\}$, but for each $i\in\mathbb{N}$ 
there exists a point $p_i=\gamma_i(s_i)$ and $y_i\in D(p_i,R_i)$ 
such that $|d\pi_i|(y_i)>2$. From this we construct a blow-up sequence of pointed Riemannian manifolds
\begin{equation*}
(M_i,h_i:=R_i^{-2} g_i, p_i)_{i\in\mathbb{N}}
\end{equation*}
which satisfies for each $i\in\mathbb{N}$ and $x\in M_i$
\begin{align*}
f_3(x,h_i)=f_3(x,R_i^{-2} g_i)\stackrel{\eqref{fkscalingeq}}{=}R_i^2\, f_3(x, g_i)\leq R_i^2\, K_i\leq \mu_i^2  \stackrel{i\to\infty}{\longrightarrow} 0
\intertext{and}
inj_{h_i}(M_i)=inj_{R_i^{-2} g_i}(M_i)=R_i^{-1}inj_{ g_i}(M_i)\geq R_i^{-1}\iota_i \geq \mu_i^{-1}\stackrel{i\to\infty}{\longrightarrow} \infty
\end{align*}
Hence, using \cite[Theorem 2.2, pp. 464-466]{anderson1989ricci}, we may extract a subsequence that converges with respect to the pointed
$C^{2,\alpha}$-sense to $(\mathbb{R}^n,g_{can},0)$. Next, for each $i\in\mathbb{N}$ we reparametrize the curve $\gamma_i$ as follows:
Let 
\begin{align*}
\widehat{\gamma}_i&: [0,\frac{L_i}{R_i}]\longrightarrow M_i\\
\widehat{\gamma}_i(s)&:=\gamma(R_i s)
\end{align*}
Then for each $i\in\mathbb{N}$ we have for all $s\in  [0,\frac{L_i}{R_i}] $
\begin{align*}
&|\dot{\widehat{\gamma}}_i(s)|_{h_i}\\
=&|\dot{\gamma}_i(R_i s)\cdot R_i|_{h_i}=R_i \cdot |\dot{\gamma}_i(R_i s)|_{R_i^{-2} g_i}=R_i\cdot R_i^{-1} |\dot{\gamma}_i(R_i s)|_{g_i}\\
=&|\dot{\gamma}_i(R_i s)|_{g_i}\stackrel{\eqref{almunit}}{\in} \left[\frac{1}{1+\beta},1+\beta \right]
\intertext{and, using normal coordinates at $\widehat{\gamma}(s)$} 
& | ^{h_i} \nabla_{\dot{\widehat{\gamma}}(s)}\dot{\widehat{\gamma}}(s)|^2_{h_i} =(h_i)_{kl}\ddot{\widehat{\gamma}}^k(s)\ddot{\widehat{\gamma}}^l(s) 
=R_i^{-2}\cdot (g_i)_{kl}\cdot \ddot{\widehat{\gamma}}^k(s)\ddot{\widehat{\gamma}}^l(s)\\
=&R_i^{-2}\cdot (g_i)_{kl}\cdot R_i^2 \cdot \ddot{\gamma}^k(R_i s)\cdot R_i^2 \cdot \ddot{\gamma}^l(R_i s)\\
=& R_i^2 \cdot (g_i)_{kl}\cdot \ddot{\gamma}^k(R_i s) \ddot{\gamma}^l(R_i s) \\
= &R_i^2\cdot | ^{g_i} \nabla_{\dot{\gamma}(R_i s)}\dot{\gamma}(R_i s)|^2_{g_i}\stackrel{\eqref{almgeod}}{\leq} R_i^2 \cdot \beta_i^2 \leq R_i^2
\end{align*}
Hence 
\begin{equation*}
\lim_{i\to\infty}\max_{[0,\frac{L_i}{R_i}]}{| ^{h_i} \nabla_{\dot{\widehat{\gamma}}}\dot{\widehat{\gamma}}|_{h_i}}=0
\end{equation*}
Using the Arzel\`a-Ascoli Theorem we conclude, that these curves converge with respect to the $C^{1,\alpha}$-sense to a geodesic which goes through the origin.
 After an eventual rotation,
we may assume that $\gamma(t)=(t,0,...,0)$. In the blow-up metric $h_i$ each point $y_i$ has a distance to $p_i$ not bigger than $1$. That means, that this point can be considered
 as a point in $B_{g_{can}}(0,2)\subseteq\mathbb{R}^n$. 
 This sequence of points will converge to a point $y\in \overline{B}_{g_{can}}(0,1)\cap \left\{x\in\mathbb{R}^n : x^1=0 \right\}$. We recall that the projection maps
 $\pi_{i}: D(\gamma_i,R_i)\longrightarrow \gamma_i([0,L_i])$ are satisfying $|d\pi_i|(y_i)>2$ by assumption. Due to the scaling invariance, this inequality is also true with respect to the blow-up metric
$h_i$. 
Since the Riemannian metrics $h_i$ converge in the $C^{2,\alpha}$-sense to the euclidean space and the curves $\gamma_i$ converge in the $C^{1,\alpha}$-sense, the maps $\pi_i$ converge in the $C^{1}$-sense to a map on the limit space, which will be denoted by $\pi$. Here we have used, that each tubular neighborhood is a diffeomorphic image of a neighborhood of the zero section
in the normal bundle on the curve $\gamma_i$ (\cite[pp. 199-200, 25. Proposition / 26. Proposition]{o1983semi}). Hence, we conclude $|d\pi|(y)\geq 2$, but the map $\pi$ is
explicitly given as $(x^1,...,x^n)\mapsto (x^1,0,...,0)$ and this map satisfies $|d\pi|\leq 1$, which yields a contradiction.
\end{proof}

We want to point that it is also possible to deduce Lemma \ref{tube} from the statement 
of \cite[Lemma 3.4, p. 272]{streets2016concentration} by use of unit-speed parametrization. On doing so, it is possible to avoid
the dependence of the constant $\beta>0$ on the diameter $D>0$.

\subsection{Forward estimates}\label{Forward estimates}
In this paragraph we show that, under certain assumptions, distances do not increase too much along the $L^2$-flow. 

Here, we prove the following estimate:
\begin{lem}\label{forwest}
Let $(M^4,g_0)$ be a closed Riemannian 4-manifold and let \linebreak 
$(M^4,g(t))_{t\in[0,1]}$ be a solution to the flow given in \eqref{eq:2.5} satisfying \eqref{eq:2.6-}, \eqref{eq:2.6}, \eqref{eq:2.6a} and \eqref{eq:2.6b}, i.e.:
\begin{align*}
\int_{M}{|Rm_{g_0}|_{g_0}^2\, dV_{g_0}}&\leq \Lambda \\
\Vert Rm_{g(t)}\Vert_{L^{\infty}(M,g(t))}&\leq K t^{-\frac{1}{2}}\\
inj_{g(t)}(M)&\geq \iota t^{\frac{1}{4}}\\
diam_{g(t)}(M)&\leq 2(1+D)
\end{align*}
for all $t\in(0,1]$. Then we have the following estimate:
\begin{equation}\label{disexforw}
d(x,y,t_2)-d(x,y,t_1)\leq C(K,\iota,D)\Lambda^{\frac{1}{2}}\left(t_2^{\frac{1}{8}}-t_1^{\frac{1}{8}}\right)^{\frac{1}{2}} +C(K,\iota,D) \left( t_2^{\frac{1}{24}}-t_1^{\frac{1}{24}} \right)
\end{equation}
for all $t_1,t_2\in [0,1]$ where $t_1<t_2$.
\end{lem}
As mentioned at the beginning of this section, we aim to use some kind of connecting curves between two points which are close to geodesics. 
These curves can be surrounded by a tube such that the projection map has bounded differential  (c.f. Lemma \ref{tube}). 

The following definition is a modification of \cite[Definition 3.1., p. 270]{streets2016concentration}. Our definition is slightly stronger in some sense because we also assume a
 stability estimate of the length of the velocity vectors along the subintervals.
 We point out that we call the following objects $\beta$-quasi-forward-geodesics and not merely
$\beta$-quasi-geodesics, as in \cite[Definition 3.1., p. 270]{streets2016concentration}. In Subsection \ref{Backward estimates} 
we introduce a time-reversed counterpart to these family of curves.
\begin{defi}\index{$\beta$-quasi-forward-geodesic}
Let $(M^n,g(t))_{t\in[t_1,t_2]}$ be a family of complete Riemannian manifolds. Given $\beta>0$ and $x,y\in M$ then we say that a family of curves $(\gamma_t)_{t\in [t_1,t_2]} : [0,1]\longrightarrow M$ is a
$\beta$-quasi-forward-geodesic connecting $x$ and $y$ if there is a constant $S>0$ so that:
\begin{enumerate}
\item For all $t\in [t_1,t_2]$ one has $\gamma_t(0)=x$ and $\gamma_t(1)=y$
\item  For all $j\in\mathbb{N}_0$ such that $t_1+jS \leq t_2$, $\gamma_{t_1+jS}$ is a length minimizing geodesic
\item  For all $j\in\mathbb{N}_0$ such that $t_1+jS \leq t_2$, and all $t\in [t_1+jS,t_1+(j+1)S)\cap [t_1,t_2]$ one has $\gamma_t= \gamma_{t_1+jS}$
\item For all $t\in [t_1,t_2]$ one has
\begin{align}
\label{eq:1.4}
 d(x,y,t)& \leq L(\gamma_t,t)\leq d(x,y,t) + \beta 
\end{align}
\item  For all $j\in\mathbb{N}_0$ such that $t_1+jS \leq t_2$, and all $t\in [t_1+jS,t_1+(j+1)S)\cap [t_1,t_2]$ one has 
\begin{align} 
\label{eq:1.6}
 \frac{1}{1+\beta}  d(x,y,t_1+jS)  \leq   \left| \dot{\gamma_t} \right|_{g(t)}&\leq  (1+\beta) d(x,y,t_1+jS)\\
\label{eq:1.5}
|^{g(t)}\nabla_{\dot{\gamma_t}}\dot{\gamma_t}|_{g(t)} &\leq \beta\, d^2(x,y,t_1+jS) 
\end{align} 
\end{enumerate}
\end{defi}

It is our aim to prove the following existence result:
\begin{lem}\label{exforw}
Let $(M^n, g(t))_{t\in [t_1,t_2]}$ a smooth family of closed Riemannian manifolds. Given $\beta>0$ and $x,y\in M$ then there exists a $\beta$-quasi-forward-geodesic connecting $x$ and $y$.
\end{lem}
\begin{rem}
The interval length $S>0$ which will be concretized along the following proof has a strong dependency on the given points $x,y\in M$,  $\beta>0$ and the flow itself. As it turns out in the proof of Lemma \ref{forwest}, this will not cause problems because estimates on the subintervals will be put together to an estimate on the entire interval $[t_1,t_2]$ via a telescope sum.
\end{rem}
\begin{proof}[Proof of Lemma \ref{exforw}]  
In order to obtain the desired existence result, we modify the proof of \cite[Lemma 3.2., p. 271]{streets2016concentration}. Let
\begin{equation}\label{DefA}
A:= \max_{t\in[t_1,t_2]}{\left\Vert g'(t) \right\Vert_{L^{\infty}(M,g(t))}}+\max_{t\in[t_1,t_2]}{\left\Vert \nabla  g'(t) \right\Vert_{L^{\infty}(M,g(t))}}
\end{equation}

At time $t_1+jS$ we choose a length minimizing geodesic
$\gamma_{t_1+jS}: [0,1]\longrightarrow M $ with respect to the metric $g(t_1+jS)$
 connecting $x$ and $y$. This curve satisfies
\begin{align}
\label{geodeq}
|\nabla_{\dot{\gamma}_{t_1+jS}}{\dot{\gamma}_{t_1+jS}}|_{g(t_1+jS)}\equiv 0
\intertext{and}
\label{constspeed}|\dot{\gamma}_{t_1+jS}|_{g(t_1+jS)} \equiv d(x,y,t_1+jS)
\end{align}

Firstly, we show that an appropriate choice of $S(\beta,x,y,g)>0$ implies \eqref{eq:1.6}. Let $v\in TM$ be an arbitrary vector and $t\in [t_1+jS,t_1+(j+1)S)\cap [t_1,t_2]$ Then, by \eqref{ap:1a}, we have 
\begin{equation}\label{eq:1.6a}
\left|\log\left(\frac{|v|^2_{g(t)}}{|v|^2_{g(t_1+jS)}} \right)\right| \leq \int_{t_1+jS}^t{\left\Vert g'(\tau)\right\Vert_{(L^{\infty}(M),g(\tau))} \ d \tau}\stackrel{\eqref{DefA}}{\leq} A S \leq \log[ (1+\beta)^2 ]
\end{equation}
Hence, we obtain the estimate
\begin{equation*}
\frac{1}{(1+\beta)^2} |\dot{\gamma}_{t_1+jS}|^2_{g(t_1+jS)} \leq|\dot{\gamma}_t|^2_{g(t)} \leq (1+\beta)^2 |\dot{\gamma}_{g(t_1+jS)}|^2_{g(t_1+jS)}
\end{equation*}
Using \eqref{constspeed} we infer \eqref{eq:1.6} from this. Next we show \eqref{eq:1.4}. Using \eqref{ap:1} we obtain
\begin{align}\label{eq:1.6.5}
\frac{\partial}{\partial t}{L(\gamma_t,t)}=\frac{\partial}{\partial t}{L(\gamma_{t_1+jS},t)} \stackrel{\eqref{DefA}}{\leq} A\cdot L(\gamma_{t_1+jS},t)= A \cdot L(\gamma_t,t) 
\end{align}
on $(t_1+jS,t_1+(j+1)S)\cap [t_1,t_2)$. This implies $\frac{\partial}{\partial t}\log L(\gamma_t,t) \leq A$, and we infer
\begin{align}
\label{eq:1.10-}
\begin{split}
d(x,y,t)\leq &L(\gamma_{t},t)=\frac{L(\gamma_{t},t) }{L(\gamma_{t_1+jS},t_1+jS)}L(\gamma_{t_1+jS},t_1+jS)\\
= & \exp\left(\log\left( \frac{L(\gamma_{t},t) }{L(\gamma_{t_1+jS},t_1+jS)}\right)\right) L(\gamma_{t_1+jS},t_1+jS) \\
= & \exp\left(\log\left( L(\gamma_{t},t) \right)- \log\left( L(\gamma_{t_1+jS},t_1+jS) \right)\right) L(\gamma_{t_1+jS},t_1+jS)\\
\leq & e^{A(t-(t_1+jS))} L(\gamma_{t_1+jS},t_1+jS)= e^{A(t-(t_1+jS))} d(x,y,t_1+jS) 
\end{split}
\end{align}
In particular, we have
\begin{equation}\label{eq:1.10}
d(x,y,t)  \leq e^{A(t_2-t_1)}  L(\gamma_{t_1},t_1)= e^{A(t_2-t_1)} d(x,y,t_1) \hspace{1cm}\forall t\in [t_1,t_2]
\end{equation}
From \eqref{eq:1.10-} we obtain for all $t\in (t_1+jS,t_1+(j+1)S)\cap [t_1,t_2] $
\begin{align*}
L(\gamma_{t},t)  \leq & d(x,y,t_1+jS)+(e^{AS}-1)d(x,y,t_1+jS)\\
\stackrel{\eqref{eq:1.10}}{\leq} & d(x,y,t_1+jS)+(e^{AS}-1) e^{A(t_2-t_1)}  d(x,y,t_1)\\
\leq  & d(x,y,t_1+jS)+\frac{\beta}{2}
\end{align*}
In order to prove \eqref{eq:1.4} it suffices to show that we can choose $S(\beta,x,y,g)>0$ small enough to ensure 
\begin{equation}
d(x,y,t_1+jS) \leq  d(x,y,t)+\frac{\beta}{2}\hspace{1cm}\forall t\in (t_1+jS,t_1+(j+1)S)\cap [t_1,t_2]
\end{equation}
From \eqref{eq:1.6a} we conclude for all $v\in TM$
\begin{equation}\label{eq:1.6b}
e^{-AS} |v|^2_{g({t_1+jS})} \leq  |v|^2_{g(t)} \leq e^{AS} |v|^2_{g(t_1+jS)} \hspace{0.5cm}\forall t\in (t_1+jS,t_1+(j+1)S) \cap [t_1,t_2]\hspace{0.5cm} 
\end{equation}
At time $t$, we choose a length minimizing geodesic $\xi: [0, d(x,y,t)]\longrightarrow M$ connecting $x$ and $y$, then:
\begin{align*}
d(x,y,t_1+jS)\leq & L(\xi,t_1+jS)=\int_{0}^{d(x,y,t)}{|\dot{\xi}(s) |_{g(t_1+jS)}\, ds }\\
\stackrel{\eqref{eq:1.6b}}{\leq} & e^{AS}  \int_{0}^{d(x,y,t)}{|\dot{\xi}(s) |_{g(t)}\, ds }=  e^{AS} L(\xi,t)= e^{AS} d(x,y,t)\\
=& d(x,y,t) +(e^{AS}-1) d(x,y,t)\\
\stackrel{\eqref{eq:1.10}}{\leq} & d(x,y,t) +(e^{AS}-1) e^{A(t_2-t_1)} d(x,y,t_1) \leq d(x,y,t) +\frac{\beta}{2} 
\end{align*}

It remains to show that, under the assumption that $S(\beta,x,y,g)>0$ is sufficiently small, estimate \eqref{eq:1.5} is also valid. From \eqref{ap:2}, \eqref{DefA} and \eqref{eq:1.6} we conclude
for each $t\in (t_1+jS,t_1+(j+1)S)\cap [t_1,t_2]$
\begin{align}
\begin{split}\label{geodesiccurvderivate}
\frac{\partial }{\partial t} \left| \nabla_{\dot{\gamma}_t}   \dot{\gamma}_t \right|_{g(t)}^2\leq  &
A \left| \nabla_{\dot{\gamma}_t}   \dot{\gamma}_t \right|_{g(t)} ^2+ 4 A C(n) d^2(x,y,t_1+jS) \left| \nabla_{\dot{\gamma}_t}   \dot{\gamma}_t \right|_{g(t)}\\
\stackrel{\eqref{eq:1.10}}{\leq}  & A \left| \nabla_{\dot{\gamma}_t}   \dot{\gamma}_t \right|_{g(t)} ^2+4 A C(n)  e^{2A(t_2-t_1)} d^2(x,y,t_1) \left| \nabla_{\dot{\gamma}_t}   \dot{\gamma}_t \right|_{g(t)}
\end{split}
\end{align}
Now let $x\in M$ be arbitrary. We assume that
\begin{align*}
&\widehat{t}:= \sup\Big\{t\in  (t_1+jS,t_1+(j+1)S)\cap [t_1,t_2) \,|\, \\
&\left| \nabla_{\dot{\gamma}_{\tau}}  \dot{\gamma}_{\tau} \right|^2_{g(\tau)}(x,\tau)\leq\min\{\beta \, \overline{d}^2,1\} \ \forall \tau\in [t_1+jS,t]\Big\}\\
&\hspace{2cm}<\min\{t_1+(j+1)S,t_2 \}
\end{align*}
where
\begin{align*}
\overline{d}&:=\min_{t\in [t_1,t_2]}{d(x,y,t)}>0
\end{align*}
Then, \eqref{geodesiccurvderivate} implies 
\begin{align*}
\frac{\partial }{\partial t} \left| \nabla_{\dot{\gamma}_t}   \dot{\gamma}_t \right|_{g(t)} ^2\leq A (1+  4 C  e^{2A(t_2-t_1)} d^2(x,y,t_1) )\text{ on }\{x\} \times [t_1+jS,\widehat{t}] 
\end{align*}
Using this, from \eqref{geodeq}, we conclude:
\begin{align*}
\min\{\beta \, \overline{d}^2,1\}&=\left| \nabla_{\dot{\gamma}_{\widehat{t}}}   \dot{\gamma}_{\widehat{t}} \right|_{g(\widehat{t})} ^2(x,\widehat{t})\\
\leq & A (\widehat{t}-(t_1+jS)) (1+  4C  e^{2A(t_2-t_1)} d^2(x,y,t_1) )\\
\leq & A S (1+  4C  e^{2A(t_2-t_1)} d^2(x,y,t_1) )\leq \frac{\min\{\beta \, \overline{d}^2,1\}}{2}
\end{align*}
which yields a contradiction, if $S(\beta,x,y,g)>0$ is small enough.
\end{proof}
Now we prove Lemma \ref{forwest}. The argumentation is based on \cite[pp. 277-280]{streets2016concentration}. 

\begin{proof}[Proof of Lemma \ref{forwest}]  
Let  $x,y\in M$ be fixed and $t_1,t_2\in [0,1]$, $t_1<t_2$. Initially, we construct an appropriate $\beta$-quasi-forward geodesic in respect of Lemma \ref{tube}. We choose
\begin{align}\label{beta_1}
\beta:=\min_{t\in[t_1,t_2]}{\beta_t}>0 
\end{align}
where 
\begin{equation*}
\beta_t:=\beta(n,diam_{g(t)}(M),f_3(M,g(t)),inj_{g(t)}(M))
\end{equation*}
is chosen according to Lemma \ref{tube} at time $t$. Next, using Lemma \ref{exforw}, we assume the existence of a $\beta$-quasi-forward-geodesic
\begin{equation*}
(\xi_t)_{t\in [t_1,t_2]}: [0,1]\longrightarrow M
\end{equation*}
connecting $x$ and $y$. It is our aim to construct
an appropriate tubular neighborhood around each $\xi_t$ applying Lemma \ref{tube}, the radii $r_t$ shall be time dependent, where $r_0=0$, when $t_1=0$. After doing this, we notice that we are able to estimate 
the integral $\int_{\xi_t}{|\text{grad }\mathcal{F}|\, d\sigma}$ from above against an average integral of $|\text{grad }\mathcal{F}|^2$ along the tube plus an error term. Each of these terms is controllable.

By construction of the $\beta$-quasi-forward-geodesic, we have a finite set of geodesics denoted by $\left(\xi_{t_1+jS}\right)_{j\in \{0,..., \lfloor\frac{t_2-t_1}{S}\rfloor\}}$,
where each of these geodesics is parametrized proportional to arc length, i.e.:
\begin{align*}
|\dot{\xi}_{t_1+jS}|_{g(t_1+jS)}\equiv d(x,y,t_1+jS)\text{ for all }j\in \{0,..., \lfloor\frac{t_2-t_1}{S}\rfloor\}
\end{align*}
we reparametrize each of these curves with respect to arc length, i.e: for each $j\in \{0,..., \lfloor\frac{t_2-t_1}{S}\rfloor\}$ let
\begin{align*}
\varphi_{t_1+jS}&: [0, d(x,y,t_1+jS)]\longrightarrow [0,1]\\
\varphi(s)&:=\frac{s}{d(x,y,t_1+jS)}
\end{align*}
and let
\begin{align*}
\gamma_{t_1+jS}&: [0, d(x,y,t_1+jS)]\longrightarrow M\\
\gamma_{t_1+jS}&:=\xi_{t_1+jS}\circ\varphi_{t_1+jS}
\end{align*}
Of course, these curves are satisfying \eqref{almlg} \eqref{almgeod} and  \eqref{almunit}. 
But we need to get sure that, for each $t\in (t_1+jS,t_1(j+1)S)\cap [t_1,t_2]$, the curve
\begin{equation*}
\gamma_t:=\xi_t\circ \varphi_{t_1+jS}: [0,d(x,y,t_1+jS)]\to M 
\end{equation*}
is also satisfying these assumptions. Here $\beta\in (0,1)$ is defined by \eqref{beta_1}. By construction,
 using \eqref{eq:1.6} for each $t\in (t_1+jS,t_1+(j+1)S)\cap [t_1,t_2]$, we have
\begin{equation*}
\frac{1}{1+\beta_t}\leq\frac{1}{1+\beta}\leq|\dot{\gamma_t}|_{g(t)}=\frac{1}{d(x,y,t_1+jS)}|\dot{\xi_t}|_{g(t)}\leq 1+\beta \leq 1+\beta_t
\end{equation*}
 and, using \eqref{eq:1.5}
\begin{equation*}
|\nabla_{\dot{\gamma_t}}\dot{\gamma_t}|_{g(t)}=\frac{1}{d(x,y,t_1+jS)^2} |\nabla_{\dot{\xi_t}}\dot{\xi_t}|_{g(t)}\leq  \beta \leq \beta_t
\end{equation*}
Thus, by Lemma \ref{tube}, for each time $t\in [jS,(j+1)S)\cap [t_1,t_2]$ the tubular neighborhood $D(\gamma_t,{\rho_t})$ is foliated by $\left(D(\gamma_t(s),{\rho_t})\right)_{s\in [0,d(x,y,t_1+jS)]}$ where
\begin{equation}\label{rhot}
\rho_t:=\mu \min\left\{inj_{g(t)}(M),f_3(M,g(t))^{-\frac{1}{2}} \right\}
\end{equation}
where $\mu>0$ is fixed and the differential of the projection map satisfies \eqref{dpi}.
 For later considerations, we assume that $\mu>0$ is also chosen  compatible to \cite[Lemma 2.7, p. 268]{streets2016concentration}.
Although we have no control on $\beta_t$, we can bound $\rho_t$ from below if we can bound $f_3(M,g(t))^{-\frac{1}{2}}$ from below 
in the view of \eqref{rhot}. 

Using \eqref{boundcurv} and \eqref{eq:2.6} we obtain for each $m\in\{1,2,3\}$:
\begin{equation}\label{derest}
\left\Vert \nabla^m Rm_{g(t)} \right\Vert_{L^{\infty}(M,g(t))}\leq C(m,K) \left(t^{-\frac{1}{2}}\right)^{\frac{2+m}{2}}=C(m,K)  t^{-\frac{2+m}{4}}
\end{equation}
and consequently
\begin{equation*}
f_3(M,g(t))\leq C(K) t^{-\frac{1}{2}}
\end{equation*}
Thus, we have for each $t\in [t_1,t_2]$
\begin{equation}\label{rhotest}
\rho_t\geq \mu \left\{\iota t^{\frac{1}{4}}, C^{-\frac{1}{2}}(K)t^{\frac{1}{4}} \right\}\geq   \mu \min\{\iota, C^{-\frac{1}{2}}(K) \}\cdot t^{\frac{7}{24}}=: R(\iota,K)\cdot t^{\frac{7}{24}}=:r_t(\iota,K)
\end{equation}
Now, we may start to estimate the change of $L(\gamma_t,t)$, where $t\in [t_1+jS,t_1+(j+1)S)\cap [t_1,t_2)$ and $j\in \left\{0,...,\lfloor \frac{t_2-t_1}{S} \rfloor \right\}$. 
From the explicit 
formula in \eqref{eq:2.5} and \eqref{derest} we conclude $|\nabla \text{ grad }\mathcal{F}_{g(t)}|_{g(t)}\leq C_2(K) t^{-\frac{5}{4}}$. Now let $p$ be an arbitrary point on the curve $\gamma_{t_1+jS}$ and $q \in D(p,r_t)$ then we obtain
\begin{equation}\label{gradFest}
|\text{grad }\mathcal{F}_{g(t)}|_{g(t)}(p) \leq | \text{grad }\mathcal{F}_{g(t)}|_{g(t)}(q)+C_3(K) r_t(\iota,K) t^{-\frac{5}{4}}
\end{equation}
In the following, we write $r_t$ instead of $r_t(\iota,K)$ and $\text{grad }\mathcal{F}$ instead of $\text{grad }\mathcal{F}_{g(t)}$. We infer:
\begin{align}
\begin{split}\label{gradtube}
|\text{grad }\mathcal{F}|_{g(t)}(p)=&\frac{\int_{D(p,r_t)}{|\text{grad }\mathcal{F}|_{g(t)}(p)\, dA}}{Area(D(p,r_t))}\\
\leq & \frac{\int_{D(p,r_t)}{\left[| \text{grad }\mathcal{F}|_{g(t)}(q)+C_3(K) r_t t^{-\frac{5}{4}} \right]\, dA}}{Area(D(p,r_t))}\\
= & \frac{\int_{D(p,r_t)}{| \text{grad }\mathcal{F}|_{g(t)} \, dA}}{Area(D(p,r_t))}+ C_3(K) R(\iota,K)  t^{\frac{7}{24}-\frac{5}{4}}\\
\leq & \frac{\left(\int_{D(p,r_t)}{| \text{grad }\mathcal{F}|_{g(t)}^2 \, dA}\right)^{\frac{1}{2}}}{Area^{\frac{1}{2}}(D(p,r_t))}+C_3(K) R(\iota,K) t^{-\frac{23}{24}}
\end{split}
\end{align}
From \cite[Lemma 2.7, p. 268]{streets2016concentration} we obtain for each $t\in [t_1,t_2]$ that
\begin{equation}\label{AreaDest}
Area(D(\gamma_t(s),r_t))\geq c r_t^{3}= c R^{3}(\iota,K)  t^{\frac{7}{8}}
\end{equation}
Inserting this estimate into \eqref{gradtube}, we infer for each $p\in \gamma_{t_1+jS}$
\begin{align}\label{gradtubeII}
\begin{split}
|\text{grad }\mathcal{F}|_{g(t)}(p) \leq & c^{-\frac{1}{2}}R^{-\frac{3}{2}}(\iota,K)  t^{-\frac{7}{16}} \left(\int_{D(p,r_t)}{| \text{grad }\mathcal{F}|_{g(t)}^2 \, dA}\right)^{\frac{1}{2}}\\
&+C_3(K) R(\iota,K) t^{-\frac{23}{24}}
\end{split}
\end{align}
Hence, on $(t_1+jS,t_1+(j+1)S)\cap [t_1,t_2)$ we have
\begin{align*}
&\frac{d}{dt}{L(\gamma_t,t)}=\frac{d}{dt}{L(\gamma_{t_1+jS},t)}\stackrel{\eqref{ap:1}}{\leq}  \int_{\gamma_{t_1+jS}}{|\text{grad }\mathcal{F}|_{g(t)}\, d\sigma}\\
\stackrel{\eqref{gradtubeII}}{\leq} & c^{-\frac{1}{2}}R^{-\frac{3}{2}}(\iota,K) t^{-\frac{7}{16}}\int_{\gamma_{t_1+jS}}{\left(\int_{D(p,r_t)}{|\text{grad }\mathcal{F}|_{g(t)}^2}\, dA \right)^{\frac{1}{2}}\, d\sigma}\\
&+C_3(K) R(\iota,K) t^{-\frac{23}{24}} L(\gamma_{t_1+jS},t)\\
\stackrel{\hphantom{\eqref{gradtubeII}}}{\leq}  & c^{-\frac{1}{2}}R^{-\frac{3}{2}}(\iota,K) t^{-\frac{7}{16}}\left(\int_{\gamma_{t_1+jS}}{\int_{D(p,r_t)}{|\text{grad }\mathcal{F}|_{g(t)}^2}\, dA \, d\sigma}\right)^{\frac{1}{2}}L^{\frac{1}{2}}(\gamma_{t_1+jS},t)\\
&+C_3(K) R(\iota,K) t^{-\frac{23}{24}} L(\gamma_{t_1+jS},t)\\
\stackrel{\eqref{coarea}}{\leq}  & c^{-\frac{1}{2}}R^{-\frac{3}{2}}(\iota,K) t^{-\frac{7}{16}}\sup_{D(\gamma_{t_1+jS},r_t)}{|d\pi|^{\frac{1}{2}}}\left(\int_{M}{|\text{grad }\mathcal{F}|_{g(t)}^2\, dV_{g(t)}}\right)^{\frac{1}{2}}L^{\frac{1}{2}}(\gamma_{t_1+jS},t)\\
&+C_3(K) R(\iota,K) t^{-\frac{23}{24}}  L(\gamma_{t_1+jS},t)\\
\stackrel{\eqref{dpi}}{\leq}  & c_2 R^{-\frac{3}{2}}(\iota,K) t^{-\frac{7}{16}}\left(\int_{M}{|\text{grad }\mathcal{F}|_{g(t)}^2\, dV_{g(t)}}\right)^{\frac{1}{2}}L^{\frac{1}{2}}(\gamma_{t_1+jS},t)\\
&+C_3(K) R(\iota,K) t^{-\frac{23}{24}} L(\gamma_{t_1+jS},t)\\
\stackrel{\hphantom{\eqref{dpi}}}{=}  & c_2 R^{-\frac{3}{2}}R(\iota,K) t^{-\frac{7}{16}}\left(\int_{M}{|\text{grad }\mathcal{F}|_{g(t)}^2\, dV_{g(t)}}\right)^{\frac{1}{2}}L^{\frac{1}{2}}(\gamma_{t},t)\\
&+C_3(K) R(\iota,K) t^{-\frac{23}{24}} L(\gamma_{t},t)
\end{align*}
Using 
\begin{equation}\label{lengthsmallerthanconst}
L(\gamma_{t},t)\stackrel{\eqref{almlg}}{\leq} d(x,y,t_1+jS)+1 \stackrel{\eqref{eq:2.6b}}{\leq} 2(1+D)+1=3+2D
\end{equation}
we conclude
\begin{align*}
 \frac{d}{dt}{L(\gamma_t,t)} \leq &   C(D) R^{-\frac{3}{2}}(\iota,K) t^{-\frac{7}{16}}\left(\int_{M}{|\text{grad }\mathcal{F}|_{g(t)}^2\, dV_{g(t)}}\right)^{\frac{1}{2}}\\
 &+C(K,D) R(\iota,K) t^{-\frac{23}{24}} 
\end{align*}
on $[t_1+jS,t_1+(j+1)S)\cap [t_1,t_2)$ where $j\in \left\{0,...,\lfloor \frac{t_2-t_1}{S} \rfloor \right\}$. Integrating this estimate along $[t_1+jS,t]$ yields:
\begin{align*}
&d(x,y,t)-d(x,y,t_1+jS)=d(x,y,t)-L(\gamma_{t_1+jS},t_1+jS)\\
\leq &L(\gamma_t,t)-L(\gamma_{t_1+jS},t_1+jS)\\
\leq & C(D) R^{-\frac{3}{2}}(\iota,K) \int_{t_1+jS}^{t}{s^{-\frac{7}{16}}\left(\int_{M}{|\text{grad }\mathcal{F}|_{g(s)}^2\, dV_{g(s)}}\right)^{\frac{1}{2}}\, ds}\\
&+C(K,D) R(\iota,K)\int_{t_1+jS}^{t}{s^{-\frac{23}{24}}\, ds}
\end{align*}
for each $t\in(t_1+jS,t_1+(j+1)S)\cap [t_1,t_2]$. In particular, we obtain for each $j\in \{0,...,\lfloor \frac{t_2-t_1}{S}\rfloor-1\}$
\begin{align*}
&d(x,y,t_1+(j+1)S)-d(x,y,t_1+jS) \\
\leq & C(D) R^{-\frac{3}{2}}(\iota,K) \int_{t_1+jS}^{t_1+(j+1)S}{s^{-\frac{7}{16}}\left(\int_{M}{|\text{grad }\mathcal{F}|_{g(s)}^2\, dV_{g(s)}}\right)^{\frac{1}{2}}\, ds}\\
&+C(K,D) R(\iota,K)\int_{t_1+jS}^{t_1+(j+1)S}{s^{-\frac{23}{24}}\, ds}
\end{align*}
and
\begin{align*}
&d(x,y,t_2)-d(x,y,t_1+\lfloor \frac{t_2-t_1}{S}\rfloor S) \\
\leq & C(D) R^{-\frac{3}{2}}(\iota,K) \int_{t_1+\lfloor \frac{t_2-t_1}{S}\rfloor S}^{t_2}{s^{-\frac{7}{16}}\left(\int_{M}{|\text{grad }\mathcal{F}|_{g(s)}^2\, dV_{g(s)}}\right)^{\frac{1}{2}}\, ds}\\
&+C(K,D) R(\iota,K)\int_{t_1+\lfloor \frac{t_2-t_1}{S}\rfloor S}^{t_2}{s^{-\frac{23}{24}}\, ds}
\end{align*}
and consequently
\begin{align*}
&d(x,y,t_2)-d(x,y,t_1)\\
=&\sum_{j=0}^{\lfloor \frac{t_2-t_1}{S}\rfloor-1}{\left[d(x,y,t_1+(j+1)S)-d(x,y,t_1+jS)\right]}\\
&+d(x,y,t_2)-d(x,y,t_1+ \lfloor \frac{t_2-t_1}{S}\rfloor S)\\
\leq &  C(D) R^{-\frac{3}{2}}(\iota,K) \int_{t_1}^{t_2}{s^{-\frac{7}{16}}\left(\int_{M}{|\text{grad }\mathcal{F}|_{g(s)}^2\, dV_{g(s)}}\right)^{\frac{1}{2}}\, ds}\\
&+C(K,D) R(\iota,K)\int_{t_1}^{t_2}{s^{-\frac{23}{24}}\, ds} \\
\leq & C(D) R^{-\frac{3}{2}}(\iota,K) \left(\int_{t_1}^{t_2}{s^{-\frac{7}{8}}\, ds}\right)^{\frac{1}{2}}\left(\int_{t_1}^{t_2}{\int_{M}{|\text{grad }\mathcal{F}|_{g(s)}^2\, dV_{g(s)}}\, ds}\right)^{\frac{1}{2}}\\
&+C(K,D) R(\iota,K)\int_{t_1}^{t_2}{s^{-\frac{23}{24}}\, ds}
\end{align*}
Using \eqref{eq:2.6-}
and \eqref{eq:2.7a} we conclude
\begin{equation*}
d(x,y,t_2)-d(x,y,t_1) \leq C(K,\iota,D)\Lambda^{\frac{1}{2}}\left(t_2^{\frac{1}{8}}-t_1^{\frac{1}{8}}\right)^{\frac{1}{2}} +C(K,\iota,D) \left( t_2^{\frac{1}{24}}-t_1^{\frac{1}{24}} \right)
\end{equation*}
\end{proof}

\subsection{Backward estimates}\label{Backward estimates}
In this subsection we reverse the ideas from Subsection \ref{Forward estimates} in order to prove that, along the  $L^2$-flow,  the distance between two points does not become too small when $t>0$ is small.
\begin{lem}\label{backest}
Let $(M^4,g_0)$ be a closed Riemannian 4-manifold and let \linebreak $(M^4,g(t))_{t\in[0,1]}$ 
be a solution to the flow given in \eqref{eq:2.5} satisfying \eqref{eq:2.6-}, \eqref{eq:2.6}, \eqref{eq:2.6a} and \eqref{eq:2.6b}, i.e.:
\begin{align*}
\int_{M}{|Rm_{g_0}|^2\, dV_{g_0}}&\leq \Lambda \\
\Vert Rm_{g(t)}\Vert_{L^{\infty}(M,g(t))}&\leq K t^{-\frac{1}{2}}\\
inj_{g(t)}(M)&\geq \iota t^{\frac{1}{4}}\\
diam_{g(t)}(M)&\leq 2(1+D)
\end{align*}
for all $t\in(0,1]$. Then we have the following estimate:
\begin{equation}\label{disexback}
d(x,y,t_2)-d(x,y,t_1)\geq - C(K,\iota,D)\Lambda^{\frac{1}{2}}\left(t_2^{\frac{1}{8}}-t_1^{\frac{1}{8}}\right)^{\frac{1}{2}} -C(K,\iota,D) \left( t_2^{\frac{1}{24}}-t_1^{\frac{1}{24}} \right)
\end{equation}
for all $t_1,t_2\in [0,1]$ where $t_1<t_2$.
\end{lem}

The notion of a $\beta$-quasi-backward-geodesic, which is introduced below, is an analogue to the notion of a $\beta$-quasi-forward-geodesic, introduced in Subsection \ref{Forward estimates}. The slight difference is that now, the minimizing geodesics are chosen at the subinterval ends:
\begin{defi}
Let $(M^n,g(t))_{t\in[t_1,t_2]}$ be a family of complete Riemannian manifolds. Given $\beta>0$ and $x,y\in M$ then we say that a family of curves $(\gamma_t)_{t\in [t_1,t_2]} : [0,1]\longrightarrow M$ is a
$\beta$-quasi-backward-geodesic connecting $x$ and $y$ if $(\gamma_t)_{t\in [t_1,t_2]}$ is a $\beta$-quasi-forward-geodesic connecting $x$ and $y$ with respect to the 
time-reversed flow $(M^n,g(t_2+t_1-t))_{t\in[t_1,t_2]}$, i.e.: there is a constant $S>0$ so that:
\begin{enumerate}
\item For all $t\in [t_1,t_2]$ one has $\gamma_t(0)=x$ and $\gamma_t(1)=y$
\item  For all $j\in\mathbb{N}_0$ such that $t_2-jS \geq t_1$, $\gamma_{t_2-jS}$ is a minimizing geodesic
\item  For all $j\in\mathbb{N}_0$ such that $t_2-jS \geq t_1$, and all $t\in (t_2-(j+1)S,t_2-j S]\cap [t_1,t_2]$ one has $\gamma_t= \gamma_{t_2-jS}$
\item For all $t\in [t_1,t_2]$ one has
\begin{align}
\label{eq:1.4_}
 d(x,y,t)& \leq L(\gamma_t,t)\leq d(x,y,t) + \beta 
\end{align}
\item For all $j\in\mathbb{N}_0$ such that $t_2-j S \geq t_1$, and all $t\in (t_2-(j+1)S,t_2-j S]\cap [t_1,t_2]$ one has 
\begin{align} 
\label{eq:1.6_}
\frac{1}{1+\beta}  d(x,y,t_1-jS) \leq   \left| \dot{\gamma}_t\right|_{g(t)} &\leq  (1+\beta)d(x,y,t_2-jS)\\
\label{eq:1.5_}
|\nabla_{\dot{\gamma}_t}\dot{\gamma}_t|_{g(t)} &\leq \beta \, d^2(x,y,t_2-jS)
\end{align} 
\end{enumerate}
\end{defi}
Applying Lemma \ref{exforw} to  $(M^n,g(t_2+t_1-t))_{t\in[t_1,t_2]}$, we infer
\begin{lem}\label{exback}
Let $(M^n, g(t))_{t\in [t_1,t_2]}$ a smooth family of closed Riemannian manifolds. Given $\beta>0$ and $x,y\in\mathbb{N}$ then there exists a $\beta$-quasi-backward-geodesic connecting $x$ and $y$.
\end{lem}

Using this concept, we prove Lemma \ref{backest}:
\begin{proof}[Proof of Lemma \ref{backest}]
The proof is analogous to Lemma \ref{forwest}. We choose $x,y\in M$ and  $t_1,t_2\in [0,1]$ where $t_1<t_2$. It is our aim to construct an appropriate backward-geodesic. As in the proof of Lemma \ref{forwest}, let
\begin{align}\label{beta_2}
\beta:=\min_{t\in[t_1,t_2]}{\beta_t}>0
\end{align}
where 
\begin{equation*}
\beta_t:=\beta(n,diam_{g(t)}(M),f_3(M,g(t)),inj_{g(t)}(M))
\end{equation*}
is defined in Lemma \ref{tube}, let $(\xi_t)_{t\in [t_1,t_2]}$ be a $\beta$-backward-geodesic,  connecting $x$ and $y$, whose existence is ensured by Lemma \ref{exback}. 
As in the proof of Lemma \ref{forwest} we use Lemma \ref{tube} to construct an
appropriate tubular neighborhood around each $\xi_t$, where $t\in [t_1,t_2]$, having a time depend radius $r_t$.

 In this situation 
we have a finite set of geodesics $\left( \xi_{t_2-jS}\right)_{j\in \{0,..., \lfloor\frac{t_2-t_1}{S}\rfloor\}}$ satisfying
\begin{equation*}
|\dot{\xi}_{t_2-jS}|_{g(t_2-jS)}\equiv d(x,y,t_2-jS)\text{ for all }{j\in \{0,..., \lfloor\frac{t_2-t_1}{S}\rfloor\}}
\end{equation*}
 Analogous to the proof of Lemma \ref{forwest}, we reparametrize these curves with respect to arc length, i.e: for each $j\in \{0,..., \lfloor\frac{t_2-t_1}{S}\rfloor\}$ we define
\begin{align*}
\varphi_{t_2-jS}&: [0, d(x,y,t_2-jS)]\longrightarrow [0,1]\\
\varphi(s)&:=\frac{s}{d(x,y,t_2-jS)}\\[1em]
\gamma_{t_2-jS}&: [0, d(x,y,t_2-jS)]\longrightarrow M\\
\gamma_{t_2-jS}&:=\xi_{t_2-jS}\circ\varphi_{t_2-jS}
\end{align*}
and for each  $t\in(t_2-(j+1)S,t_2-jS]\cap [t_1,t_2]$ we define
\begin{equation*}
\gamma_t:=\xi_t\circ \varphi_{t_2-jS}: [0,d(x,y,t_2-jS)]\to M 
\end{equation*}
so that, for each $t\in [t_1,t_2]$ the curve $\gamma_t$ satisfies \eqref{almlg} \eqref{almgeod} and  \eqref{almunit} with respect to $\beta_t$. 
Hence, following Lemma \ref{tube}, at each time $t\in (t_2-(j+1)S,t_2-jS]\cap [t_1,t_2]$ the tubular
neighborhood $D(\gamma_t,\rho_t)$ around $\gamma_t$ is foliated by $\left(D(\gamma_t(s),{\rho_t})\right)_{s\in [0,d(x,y,t_2-jS)]}$ where $\rho_t:=\mu \min\left\{inj_{g(t)}(M),f_3(M,g(t))^{-\frac{1}{2}} \right\}$, again
 $\mu>0$ shall also satisfy the requirements of Lemma \ref{tube}. Using the same arguments as in the proof of Lemma \ref{forwest} we also obtain \eqref{rhotest} and \eqref{gradFest}, i.e.: 
\begin{equation*}
\rho_t \geq R(\iota,K) t^{\frac{7}{24}}=:r_t(\iota,K)\text{ for each }t\in [t_1,t_2]
\end{equation*}
 and 
\begin{equation*}
|\text{grad }\mathcal{F}|_{g(t)}(p)\leq | \text{grad }\mathcal{F}|_{g(t)}(q)+C_3(K) r_t(\iota,K) t^{-\frac{5}{4}}
\end{equation*}
for each $p\in \gamma_{t}=\gamma_{t_2-jS}$ and $q \in D(p,r_t)$ where $t\in (t_2-(j+1)S,t_2-jS]\cap [t_1,t_2] $ and $j\in \{0,...,\lfloor\frac{t_2-t_1}{S}\rfloor \}$. From this we also obtain \eqref{gradtube}, i.e.:
\begin{align*}
| \text{grad }\mathcal{F}|_{g(t)}(p)\leq \frac{\left(\int_{D(p,r_t)}{| \text{grad }\mathcal{F}|_{g(t)}^2(q) \, dA(q)}\right)^{\frac{1}{2}}}{Area^{\frac{1}{2}}(D(p,r_t))}+C_3(K) R(\iota,K) t^{-\frac{23}{24}}
\end{align*}
Using \cite[Lemma 2.7, p. 268]{streets2016concentration} we obtain \eqref{AreaDest}, i.e.: 
\begin{equation*}
Area(D(\gamma_t(s)),r_t)\geq c r_t^{3}= c R^{3}t^{\frac{7}{8}}
\end{equation*}
for all $t\in (t_2-(j+1)S,t_2-jS]\cap [t_1,t_2]$. Hence, for each $j\in \{0,...,\lfloor\frac{t_2-t_1}{S}\rfloor \}$ we infer on $ (t_2-(j+1)S,t_2-jS)\cap (t_1,t_2]$ the following estimate
\begin{align*}
&\frac{d}{dt}{L(\gamma_t,t)}=\frac{d}{dt}{L(t_2-jS,t)}\stackrel{\eqref{ap:1}}{\geq} -\int_{\gamma_{t_2-jS}}{|\text{grad }\mathcal{F}|_{g(t)}\, d\sigma}\\
\geq & - c^{\frac{1}{2}} R^{-\frac{3}{2}}(\iota,K) t^{-\frac{7}{16}}\int_{\gamma_{t_2-jS}}{\left(\int_{D(p,r_t)}{|\text{grad }\mathcal{F}|_{g(t)}^2}\, dA \right)^{\frac{1}{2}}\, d\sigma}\\
&-C_3(K) R(\iota,K) t^{-\frac{23}{24}} L(\gamma_{t_2-jS},t)\\
\geq  & -c^{\frac{1}{2}} R^{-\frac{3}{2}}(\iota,K) t^{-\frac{7}{16}}\left(\int_{\gamma_{t_2-jS}}{\int_{D(p,r_t)}{|\text{grad }\mathcal{F}|_{g(t)}^2}\, dA \, d\sigma}\right)^{\frac{1}{2}}L^{\frac{1}{2}}(\gamma_{t_2-jS},t)\\
&-C_3(K) R(\iota,K) t^{-\frac{23}{24}} L(\gamma_{t_2-jS},t)\\
\stackrel{\eqref{coarea}}{\geq}  & -c^{\frac{1}{2}} R^{-\frac{3}{2}}(\iota,K) t^{-\frac{7}{16}}\sup_{D(\gamma_{t_2-jS},r_t)}{|d\pi|^{\frac{1}{2}}}\left(\int_{M}{|\text{grad }\mathcal{F}|_{g(t)}^2 dV_{g(t)}}\right)^{\frac{1}{2}}L^{\frac{1}{2}}(\gamma_{t_2-jS},t)\\
&-C_3(K) R(\iota,K) t^{-\frac{23}{24}} L(\gamma_{t_2-jS},t)\\
\stackrel{\eqref{dpi}}{\geq}   & -c_2 R^{-\frac{3}{2}}(\iota,K) t^{-\frac{7}{16}}\left(\int_{M}{|\text{grad }\mathcal{F}|_{g(t)}^2\, dV_{g(t)}}\right)^{\frac{1}{2}}L^{\frac{1}{2}}(\gamma_{t_2-jS},t)\\
&-C_3(K) R(\iota,K)  t^{-\frac{23}{24}} L(\gamma_{t_2-jS},t)\\
\geq  &-  C(D)R^{-\frac{3}{2}}(\iota,K) t^{-\frac{7}{16}}\left(\int_{M}{|\text{grad }\mathcal{F}|_{g(t)}^2\, dV_{g(t)}}\right)^{\frac{1}{2}}-C(K,D) R(\iota,K) t^{-\frac{23}{24}} 
\end{align*}
Here we have used the fact that $\gamma_t$ is nearly length minimizing and that the diameter is bounded (cf. \eqref{lengthsmallerthanconst}).
 By integration along  $[t,t_2-jS]$ we conclude for each $t\in(t_2-(j+1)S,t_2-jS]\cap [t_1,t_2]$
\begin{align*}
&d(x,y,t_2-jS)-d(x,y,t)=L(\gamma_{t_2-jS},t_2-jS)-d(x,y,t)\\
\geq &L(\gamma_{t_2-jS},t_2-jS)-L(\gamma_{t},t)\\
\geq &-C(D) R^{-\frac{3}{2}}(\iota,K)\int_{t}^{t_2-jS}{s^{-\frac{7}{16}}\left(\int_{M}{|\text{grad }\mathcal{F}|_{g(s)}^2\, dV_{g(s)}}\right)^{\frac{1}{2}}\, ds}\\
&-C(K,D) R(\iota,K) \int_{t}^{t_2-jS}{s^{-\frac{23}{24}}\, ds}
\end{align*}
In particular, we have for each $j\in \{0,...,\lfloor \frac{t_2-t_1}{S}\rfloor-1 \}$
\begin{align*}
&d(x,y,t_2-jS)-d(x,y,t_2-(j+1)S) \\
\geq &-C(D) R^{-\frac{3}{2}}(\iota,K)\int_{t_2-(j+1)S}^{t_2-jS}{s^{-\frac{7}{16}}\left(\int_{M}{|\text{grad }\mathcal{F}|_{g(s)}^2\, dV_{g(s)}}\right)^{\frac{1}{2}}\, ds}\\
&-C(K,D) R(\iota,K)\int_{t_2-(j+1)S}^{t_2-jS}{s^{-\frac{23}{24}}\, ds}
\end{align*}
and also
\begin{align*}
&d(x,y,t_2-\lfloor \frac{t_2-t_1}{S}\rfloor S)-d(x,y,t_1) \\
\geq - &C(D) R^{-\frac{3}{2}}(\iota,K)\int_{t_1}^{t_2-\lfloor \frac{t_2-t_1}{S}\rfloor S}{s^{-\frac{7}{16}}\left(\int_{M}{|\text{grad }\mathcal{F}|_{g(s)}^2\, dV_{g(s)}}\right)^{\frac{1}{2}}\, ds}\\
&-C(K,D) R(\iota,K)\int_{t_1}^{t_2-\lfloor \frac{t_2-t_1}{S}\rfloor S}{s^{-\frac{23}{24}}\, ds}
\end{align*}
and finally
\begin{align*}
&d(x,y,t_2)-d(x,y,t_1)\\
=&\sum_{j=0}^{\lfloor \frac{t_2-t_1}{S}\rfloor-1}{\left[d(x,y,t_2-jS)-d(x,y,t_2-{(j+1)S})\right]}\\
&+d(x,y,t_2-\lfloor \frac{t_2-t_1}{S}\rfloor S)-d(x,y,t_1)\\
\geq & -C(D)R^{-\frac{3}{2}}(\iota,K)\int_{t_1}^{t_2}{s^{-\frac{7}{16}}\left(\int_{M}{|\text{grad }\mathcal{F}|_{g(s)}^2\, dV_{g(s)}}\right)^{\frac{1}{2}}\, ds}\\
&-C(K,D) R(\iota,K)\int_{t_1}^{t_2}{s^{-\frac{23}{24}}\, ds} \\
\geq &-C(K,D) R^{-\frac{3}{2}}(\iota,K)\left(\int_{t_1}^{t_2}{s^{-\frac{7}{8}}\, ds}\right)^{\frac{1}{2}}\left(\int_{t_1}^{t_2}{\int_{M}{|\text{grad }\mathcal{F}|_{g(s)}^2\, dV_{g(s)}}\, ds}\right)^{\frac{1}{2}}\\
&-C(K,D) R(\iota,K)\int_{t_1}^{t_2}{s^{-\frac{23}{24}}\, ds}
\end{align*}
we infer
\begin{equation*}
d(x,y,t_2)-d(x,y,t_1)\geq  - C(K,\iota,D)\Lambda^{\frac{1}{2}}\left(t_2^{\frac{1}{8}}-t_1^{\frac{1}{8}}\right)^{\frac{1}{2}} -C(K,\iota,D) \left( t_2^{\frac{1}{24}}-t_1^{\frac{1}{24}} \right)
\end{equation*}
\end{proof}

Finally, \eqref{disexforw} and \eqref{disexback} together imply \eqref{disex}, which finishes the proof of  Theorem \ref{th:2a}. Using Theorem \ref{th:2a},
the following result

\begin{corollary}\label{GHclose}
Let $(M^4,g(t))_{t\in[0,1]}$, where Let $M^4$ is a closed Riemannian 4-manifold, be a solution to \eqref{eq:2.5} satisfying the assumptions, \eqref{eq:2.6-}, \eqref{eq:2.6}, \eqref{eq:2.6a} and \eqref{eq:2.6b},
then for each $k\in\mathbb{N}$ there exists $j(k,\Lambda,K,\iota,D)\in\mathbb{N}$ such that
\begin{align*}
 d_{GH}((M,d_{g}),(M,d_{g(t)}))<\frac{1}{k}
\end{align*}
for all $t\in [0,1/j]$
\end{corollary}
is a consequence of the following Lemma
\begin{lem}\label{distGH}
Let $M^n$ be a closed manifold. Given two metrics $g_1$ and $g_2$ on $M$ satisfying 
\begin{equation*}
\sup_{x,y\in M}{|d_{g_1}(x,y)-d_{g_2}(x,y)|}< \epsilon
\end{equation*}
then we have 
\begin{equation*}
d_{GH}((M,d_{g_1}),(M,d_{g_2}))<\frac{\epsilon}{2}
\end{equation*}
\end{lem}
\begin{proof}
The set $\mathfrak{R}:=\{(x,x)\in M\times M\, |\,x\in M \}$ is a correspondence between $M$ and $M$ itself (cf. \cite[Definition 7.3.17., p. 256]{burago2001course}) and the distorsion of $\mathfrak{R}$ (cf. \cite[Definition 7.3.21., p. 257]{burago2001course}) is :
\begin{equation*}
\text{dis}\,\mathfrak{R}=\sup_{x,y\in M}{|d_{g_1}(x,y)-d_{g_2}(x,y)|}< \epsilon
\end{equation*}
From \cite[Theorem 7.3.25., p. 257]{burago2001course} we obtain
\begin{equation*}
d_{GH}((M,d_{g_1}),(M,d_{g_2}))\leq \frac{1}{2}\text{dis}\,\mathfrak{R}<\frac{1}{2} \epsilon
\end{equation*}
\end{proof}

\section{Proof of Theorem \ref{th:1}}\label{sec:3}
In this section we prove Theorem \ref{th:1} using Corollary \ref{GHclose}. The conditions \eqref{eq:2.6-}, \eqref{eq:2.6}, \eqref{eq:2.6a} and \eqref{eq:2.6b} are ensured by the following result
\begin{theorem}\label{strflow}{(cf. \cite[Theorem 1.8, p. 260]{streets2016concentration})}\label{th:1a}
Given $\delta\in (0,1)$, there are constants $\epsilon(\delta),\iota(\delta),A(\delta)>0$ so that if $(M^4,g_0)$ is a closed Riemannian manifold satisfying the following conditions
\begin{align}
\notag\mathcal{F}_{g_0}&\leq \epsilon\\
\label{eq:1.2}
Vol_{g_0}(B_{g_0}(x,r))&\geq \delta \omega_4 r^4\hspace{1cm}\forall x\in M, r\in [0,1]
\end{align}
then the flow given in \eqref{eq:2.5} with initial metric $g_0$ has a solution on $[0,1]$ and we have the following estimates:
\begin{align*}
\left\Vert Rm_{g(t)} \right\Vert_{L^{\infty}(M,g(t))}&\leq A \mathcal{F}^{\frac{1}{6}}_{g(t)}t^{-\frac{1}{2}} \\
inj_{g(t)}(M)&\geq \iota t^{\frac{1}{4}} \\
diam_{g(t)}(M)&\leq 2(1 +diam_{g(0)}(M))
\end{align*}
for all $t\in (0,1]$.
\end{theorem}

From these estimates we may conclude the following precompactness result, at first
\begin{corollary}\label{curvsmallprecompactness}
Given $D,\delta>0$. Then there exists $\epsilon(\delta)>0$ so that the space $\mathcal{M}^4(D,\delta,\epsilon(\delta))$ \index{$\mathcal{M}^4(D,\delta,\epsilon(\delta))$} which consists of the set of all closed Riemannian $4$-manifolds $(M,g)$ satisfying
\begin{align*}
diam_g(M)&\leq D\\
Vol_{g}(B_{g}(x,r))&\geq \delta \omega_4 r^4\hspace{1cm}\forall x\in M, r\in [0,1] \\
\Vert Rm_g\Vert_{L^2(M,g)}&\leq \epsilon^2
\end{align*}
equipped with the Gromov-Hausdorff topology, is precompact.
\end{corollary}
\begin{proof}
Let $(M,g)$ be an element in $\mathcal{M}^4(D,\delta,\epsilon(\delta))$.
Using Theorem \ref{strflow} we know that the $L^2$-flow with initial metric $g$ exists on the time interval $[0,1]$. Together with \eqref{eq:2.7} we ensure that the following estimates are valid 
\begin{align*}
\left\Vert Rm_{g(t)} \right\Vert_{L^{\infty}(M,g(t))}&\leq A \mathcal{F}^{\frac{1}{6}}_{g(t)}t^{-\frac{1}{2}}\stackrel{\eqref{eq:2.7}}{\leq} A \mathcal{F}^{\frac{1}{6}}_{g(0)}t^{-\frac{1}{2}}\leq  t^{-\frac{1}{2}}  \\
diam_{g(t)}(M)&\leq 2(1 +D)
\end{align*}
Hence, from the Bishop-Gromov comparison principle (cf. \cite[Lemma 36. p. 269]{petersen2006riemannian}) we infer
\begin{equation}\label{VolBound}
Vol_{g(0)}(M)\stackrel{\eqref{VolInvariant}}{=}Vol_{g(1)}(M)=Vol_{g(1)}{B_{g(1)}(x,2(1+D))}\leq V_0(D)
\end{equation}

Now, let $\{ x_1, ..., x_{N(M,g)} \} \subseteq M$ be a maximal $r$-separated set, which implies that $\{ x_1, ..., x_N \}$
 is an $r$-net. In this situation the balls 
\begin{equation*}
B_{g}(x_1,\frac{r}{2}),..., B_{g}(x_N,\frac{r}{2})
\end{equation*}
are mutually disjoint and the balls  $B_{g}(x_1,r),..., B_{g}(x_N,r)$ cover  $M$.
Using the non-collapsing assumption (cf. \eqref{eq:1.2}), we infer
\begin{align*}
N \omega_4 \delta \left( \frac{r}{2} \right)^n \leq  &\sum_{k=1}^N{ Vol_{g}(B_{g}(x_k,\frac{r}{2}}))\\
= &Vol_{g}(\bigcup_{k=1}^N{B_{g}(x_k,\frac{\epsilon}{2}}))\leq Vol_{g}(M)\stackrel{\eqref{VolBound}}{\leq} V_0(D)
\end{align*}
This implies that the number of elements in such an $r$-net is bounded from above by a natural number $N(r,\delta,D)$. The assertion follows from \cite[Theorem 7.4.15, p. 264]{burago2001course}.
\end{proof}

\begin{proof}[Proof of Theorem \ref{th:1}]
As in the proof of Corollary \ref{curvsmallprecompactness}, we know that for each $i\in \mathbb{N}$ the $L^2$-flow with initial metric $g_i$ exists on $[0,1]$ and that this flow satisfies the following estimates
\begin{align}\label{curvtozero}
\begin{split}
\left\Vert Rm_{g_i(t)} \right\Vert_{L^{\infty}(M,g_i(t))}&\leq A \mathcal{F}^{\frac{1}{6}}_{g_i(t)}t^{-\frac{1}{2}}\stackrel{\eqref{eq:2.7}}{\leq} A \left(\frac{1}{i} \right)^{\frac{1}{6}} t^{-\frac{1}{2}}\leq  t^{-\frac{1}{2}}  \\
inj_{g_i(t)}(M)&\geq \iota t^{\frac{1}{4}} \\
diam_{g_i(t)}(M)&\leq 2(1 +D)
\end{split}
\end{align}
for all $t\in (0,1]$.
Using Corollary \ref{GHclose}, we may choose a monotone decreasing sequence $(t_j)_{j\in\mathbb{N}}\subseteq (0,1]$ that converges to zero and that satisfies
\begin{equation*}
d_{GH}((M_i,g_i),(M_i,g_i(t_j)))<\frac{1}{3j}\hspace{1cm}\forall i,j\in\mathbb{N}
\end{equation*}
Estimate \eqref{boundcurv} implies, that for each $m\in \mathbb{N}$
\begin{equation}
\left\Vert \nabla^m Rm_{g_i(t_j)} \right\Vert_{L^{\infty}(M_i,g_i(t_j))}\leq C(m)  t_j^{-\frac{2+m}{4}}\hspace{1cm}\forall i,j\in\mathbb{N}
\end{equation}
As in the proof of Corollary \ref{curvsmallprecompactness} we also have
\begin{equation*}
v_0(D,\delta)\leq Vol_{g_i(t_j)}{(M_i)}=Vol_{g_i(1)}{(M_i)}\leq V_0(D)
\end{equation*}
where we have used the non-collapsing assumption in order to prove the lower bound.
Hence, at each time $t_j$, we are able to apply \cite[Theorem 2.2, pp. 464-466]{anderson1989ricci}  to the sequence of manifolds $(M_i,g_i(t_j))_{i\in\mathbb{N}}$, i.e.: for all $j\in\mathbb{N}$ there exists
a subsequence $(M_{i(j,k)},g_{i(j,k)}(t_j))_{k\in\mathbb{N}}$ converging in the $C^{m,\alpha}$-sense, where $m\in\mathbb{N}$ is arbitrary, to a smooth manifold $(N_j,h_j)$
as $k$ tends to infinity. We may assume that the
selection process is organized so that each sequence $(M_{i(j,k)},g_{i(j,k)}(t_j))_{k\in\mathbb{N}}$ is a subsequence of $(M_{i(j-1,k)},g_{i(j-1,k)}(t_j))_{k\in\mathbb{N}}$. 
The smooth convergence together with \eqref{curvtozero} implies $Rm_{h_j}\equiv 0$ for each $j\in\mathbb{N}$. 

In order to apply \cite[Theorem 2.2, pp. 464-466]{anderson1989ricci} to the sequence $(N_j,h_j)_{j\in\mathbb{N}}$, we need an argument for a uniform lower bound on the injectivity radius because the injectivity radius estimate in \eqref{curvtozero} is not convenient. 
To overcome this issue, we recall that the volume of balls does not decay to quickly along the flow (cf. Lemma \ref{L2volestimate}) and the convergence is smooth. So, the volume of suitable balls is well-controlled
from below. Since $(N_j,h_j)$ is flat, we are able to apply \cite[Theorem 4.7, pp. 47-48]{cheeger1982finite}, which yields a uniform lower bound on the injectivity radius for each $(N_j,h_j)$. Hence, there exists a subsequence
of $(N_j,h_j)_{j\in\mathbb{N}}$ that converges in the $C^{\infty}$-sense, to a flat manifold $(M,g)$. 
 Finally we need to get sure
 that $(M_i,g_i)_{i\in\mathbb{N}}$ contains a subsequence that also converges to $(M,g)$, at least in the Gromov-Hausdorff sense.
For each $m\in\mathbb{N}$, we choose $j(m)\geq m$ so that
\begin{align*}
d_{GH}((M,g),(N_{j(m)},h_{j(m)}))\leq \frac{1}{3m}
\end{align*}
and $k(m)\in\mathbb{N}$ so that
\begin{equation*}
d_{GH}((N_{j(m)},h_{j(m)}),(M_{i(j(m),k(m))},g_{i(j(m),k(m))}(t_{j(m)}))\leq \frac{1}{3m}
\end{equation*}
This implies
\begin{align*}
&d_{GH}((M,g),(M_{i(j(m),k(m))},g_{i(j(m),k(m))})\\
\leq &d_{GH}((M,g),(N_{j(m)},h_{j(m)}))\\
&+d_{GH}((N_{j(m)},h_{j(m)}),(M_{i(j(m),k(m))},g_{i(j(m),k(m))}(t_{j(m)}))\\
&+d_{GH}((M_{i(j(m),k(m))},g_{i(j(m),k(m))}(t_{j(m)}),(M_{i(j(m),k(m))},g_{i(j(m),k(m))})))\\
\leq & \frac{1}{3m}+\frac{1}{3m}+\frac{1}{3j(m)}\leq  \frac{1}{3m}+\frac{1}{3m}+\frac{1}{3m}=\frac{1}{m}
\end{align*}
and this implies, that the sequence $(M_{i(j(m),k(m))},g_{i(j(m),k(m))})_{m\in\mathbb{N}}$ converges with respect to the Gromov-Hausdorff topology to $(M,g)$ as $m$ tends to infinity.
\end{proof}

\section{Proof of Theorem \ref{th:1aa}}\label{prth12}

In order to apply Theorem \ref{th:2a} to the situation in Theorem \ref{th:1aa} we give a proof of the following existence result

\begin{theorem}\label{existenceII}
\label{th:4.2}
Let $D,\Lambda>0$. Then there are universal constants $\delta\in(0,1)$, $K>0$ and constants $\epsilon(\Lambda),T(\Lambda)>0$ satisfying the following property: Let $(M,g)$ be a closed Riemannian $4$-manifold satisfying
\begin{align*}
diam_g(M)&\leq D\\
\Vert Rm_g \Vert_{L^2(M,g)}&\leq \Lambda\\
Vol_g(B_g(x,r))&\geq \delta \omega_4 r^4\hspace{1cm}\forall x\in M,\, r\in [0,1]\\
\Vert \mathring{Rc}_g \Vert_{L^2(M,g)}&\leq \epsilon
\end{align*}
then the $L^2$-flow exists on $[0,T]$, and we have the following estimates:
\begin{align}\label{eq:Rminj}
\begin{split}
\Vert Rm_{g(t)}\Vert _{L^{\infty}(M,g(t))}&\leq K t^{-\frac{1}{2}}\\
inj_{g(t)}(M)&\geq t^{\frac{1}{4}}
\end{split}
\end{align}
and
\begin{equation}\label{diamestimate}
diam_{g(t)}(M)\leq 2(1+D)
\end{equation}
for all $t\in (0,T]$.
\end{theorem}

We point out that J. Streets has proved this result as a part of the proof of \cite[Theorem 1.21]{streets2016concentration} (cf. \cite[pp. 285-287]{streets2016concentration}). 
For sake of completeness, we also want to give a proof here, under the viewpoint
 of the dependence of $\epsilon$ and $T$ on given parameters and that \eqref{diamestimate} is also satisfied.

\begin{proof}
We follow the lines of \cite[pp. 285-286]{streets2016concentration}, giving further details.
At first, we allow $\delta\in(0,1)$ and $K>0$ to be arbitrary but fixed. Along the proof, we concretize these constants. We argue by contradiction.

Suppose, there is a sequence of 
closed Riemannian $4$-manifolds $(M_i,g_i)_{i\in\mathbb{N}}$ so that for all $i\in\mathbb{N}$ we have the following estimates:
\begin{align*}
\int_{M_i}{|Rm_{g_i}|_{g_i}^2\, dV_{g_i}}&\leq \Lambda\\
Vol_{g_i}(B_{g_i}(x,r))&\geq \delta \omega_4 r^4\hspace{1cm}\forall r\in [0,1]\\
\end{align*}
and 
\begin{equation*}
\int_{M_i}{|\mathring{Rc}_{g_i}|_{g_i}^2\, dV_{g_i}}\leq \frac{1}{i}
\end{equation*}
but the estimates \eqref{eq:Rminj} hold on a maximal interval $[0,T_i]$ where $\lim_{i\to\infty}{T_i}=0$. We consider the following sequence of rescaled metrics:
\begin{equation*}
\overline{g}_i(t):=T_i^{-\frac{1}{2}}g_i(T_i t)
\end{equation*}
Then, for each $i\in\mathbb{N}$ the solution of the $L^2$-flow exists on $[0,1]$ and satisfies:
\begin{align}\label{eq:Rminjscaled}
\begin{split}
\Vert Rm_{\overline{g}_i(t)}\Vert_{L^{\infty}(M,\overline{g_i}(t))}&=T_i^{\frac{1}{2}}\Vert Rm_{g_i(T_i t)}\Vert_{L^{\infty}(M,g_i(T_i t))}\leq T_i^{\frac{1}{2}} K (T_it)^{-\frac{1}{2}}=K t^{-\frac{1}{2}}\\
inj_{\overline{g_i}(t)}(M_i)&=T_i^{-\frac{1}{4}}inj_{g_i(T_i t)}\geq T_i^{-\frac{1}{4}}(T_i t)^{\frac{1}{4}}=t^{\frac{1}{4}}
\end{split}
\end{align}
on $[0,1]$, which means that the estimates \eqref{eq:Rminj} are formally preserved under this kind of rescaling.

By assumption, for each $i\in\mathbb{N}$, one of the inequalities in \eqref{eq:Rminjscaled} is an equality at time $t=1$.
In respect of the generalized Gauss-Bonnet Theorem (cf. \cite[Appendix A]{1504.02623}), i.e.:
\begin{align}\label{gaussbonnet}
\begin{split}
\int_M{|Rm|^2\, dV_{g}}&=c_0\pi^2 \chi(M)+4\int_M{|Rc|^2\, dV_{g}}-\int_M{R^2\, dV_{g}}\\
&=c_0\pi^2 \chi(M)+4\int_M{|\mathring{Rc}|^2\, dV_{g}}
\end{split}
\end{align}
where we have used
\begin{align*}
|\mathring{Rc}|^2=&\left|Rc-\frac{1}{4}R g\right|^2=\left|Rc\right|^2-\frac{1}{2}\langle Rc ,R g\rangle+\frac{1}{16}R^2 |g|^2\\
=&\left|Rc\right|^2-\frac{1}{2} R\, tr (Rc)+\frac{1}{4} R^2  =\left|Rc\right|^2-\frac{1}{2}R^2+\frac{1}{4} R^2\\
=&\left|Rc\right|^2-\frac{1}{4}R^2
\end{align*}
we introduce the following functional
\begin{equation*}
\mathcal{G}_g:=\int_M{|\mathring{Rc}_g|_g^2\, dV_g}
\end{equation*}
From \eqref{gaussbonnet} and \cite[4.10 Definition, p. 119]{besse2007einstein} we infer
\begin{equation*}
\text{grad }\mathcal{F}\equiv 4 \text{ grad }\mathcal{G}
\end{equation*}
As in the proof of Lemma \ref{energydifferencelem} we obtain for each  $i\in\mathbb{N}$ and $t\in [0,T_i]$
\begin{align*}
\mathcal{G}_{g_i(0)}-\mathcal{G}_{g_i(t)}=\int_{0}^t{\int_{M_i}{|\text{grad }\mathcal{G}_{g_i(s)}|_{g_i(s)}^2\, dV_{g_i(s)}}\, ds }\geq 0
\end{align*}
which implies $\mathcal{G}_{g_i(t)}\leq \frac{1}{i}$ for each $i\in\mathbb{N}$ and $t\in [0,T_i]$. Due to the scale invariance of the functional $\mathcal{G}$, we have in particular
\begin{equation*}
\mathcal{G}_{\overline{g}_i(1)}\leq \frac{1}{i}\hspace{0.5cm}\text{for all }i\in\mathbb{N}
\end{equation*}
As already stated, \eqref{eq:Rminjscaled} implies 
\begin{equation*}
\Vert Rm_{\overline{g}_i(1)}\Vert_{L^{\infty}(M,\overline{g}_i(1))}=K \hspace{0.5cm}\text{or}\hspace{0.5cm}inj_{\overline{g}_i}(M_i)=1
\end{equation*}
for each $i\in\mathbb{N}$. 

At first, we assume that there is a subsequence  $(M_{i},\overline{g}_{i})_{i\in\mathbb{N}}$ (we do not change the index) 
satisfying
\begin{align*}
\begin{cases}
&\Vert Rm_{\overline{g}_{i}(1)}\Vert_{L^{\infty}(M_{i},\overline{g}_{i}(1))}=K\\
& inj_{\overline{g}_{i}(1)}(M_{i})\geq 1
\end{cases}
\end{align*}
for each $i\in\mathbb{N}$. Using the compactness, for each $j\in\mathbb{N}$ we may choose a point $p_{i}\in M_{i}$ satisfying $|Rm_{\overline{g}_{i}(1)}(p_{i})|_{\overline{g}_{i}(1)}=K$. From \cite[Corollary 1.5, p. 42]{streets2013long} 
we conclude that there exists a subsequence of manifolds, also index by $i$, and a complete pointed $4$-manifold $(M_{\infty},p_{\infty})$ together with
a $1$-parametrized family of Riemannian metrics $(g_{\infty}(t))_{t\in [1/2,1]}$ on $M_{\infty}$ such that for each $t\in [1/2,1]$ 
\begin{equation*}
(M_{i},\overline{g}_{i}(t),p_{i})\stackrel{i\to\infty}{\longrightarrow} (M_{\infty},g_{\infty}(t),p_{\infty})
\end{equation*}
in the sense of $C^{\infty}$-local submersions (cf. \cite[Definition 2.4, p. 45]{streets2013long}), and 
\begin{align*}
\Vert Rm_{g_{\infty}(1)}\Vert_{L^{\infty}(M_{\infty},g_{\infty}(1))}&=|Rm_{g_{\infty}(1)}(p_{\infty})|_{g_{\infty}(1)}=K
\intertext{as well as, using \cite[Theorem]{sakai1983continuity}}
inj_{g_{\infty}(1)}(M_{\infty})&\geq 1
\end{align*}
Since $\lim_{i\to\infty}{\mathcal{G}_{\overline{g}_i(1)}}=0$ we conclude that $(M_{\infty},g_{\infty}(1),p_{\infty})$ needs to be an Einstein manifold satisfying
\begin{equation}\label{energyboundatinfty}
\int_{M_{\infty}}{|Rm_{g_{\infty}(1)}|_{g_{\infty}(1)}^2\, dV_{{g_{\infty}(1)}}}\leq \Lambda
\end{equation}
 In particular, \cite[Proposition 7.8, p. 125]{lee1997riemannian} implies that the scalar curvature is constant. On the other hand,
 from the non-collapsing condition and \eqref{VolInvariant} we obtain that ${Vol_{\overline{g}_i(1)}(M_i)}$ tends to infinity as $i\in\mathbb{N}$ tends to infinity. Then, estimate \eqref{energyboundatinfty} implies
that the scalar curvature needs to vanish on $(M_{\infty},\overline{g}_{\infty}(1))$, hence $(M_{\infty},\overline{g}_{\infty}(1))$ is a Ricci-flat manifold. From Lemma \ref{ricflatimplboundcurv} we obtain
\begin{equation*}
\Vert Rm_{\overline{g}_{\infty}(1)}\Vert_{L^{\infty}(M_{\infty},\overline{g}_{\infty}(1))}\leq C
\end{equation*}
where $C$ is a universal constant, since the space dimension is fixed and the injectivity radius is bounded from below by 1.
Choosing $K=C+1$ we obtain a contradiction to $|Rm_{\overline{g}_{\infty}(1)}(p_{\infty})|_{\overline{g}_{\infty}(1)}=K$. 
This finishes the part of the proof that $\Vert Rm_{g_i(T_i)}\Vert_{L^{\infty}(M,g_i(T_i))}=K T_i^{-\frac{1}{2}}$ can only be valid for a finite number of $i\in\mathbb{N}$.

Now we assume that, after taking a subsequence, we are in the following situation
\begin{align*}
\begin{cases}
&\Vert Rm_{\overline{g}_{i}(1)}\Vert_{L^{\infty}(M_{i},\overline{g}_{i}(1))}\leq K\\
& inj_{\overline{g}_{i}(1)}(M_{i}) = 1
\end{cases}
\end{align*}
Then, the non-collapsing assumption of the initial sequence implies the following non-collapsing condition concerning the rescaled metrics
\begin{equation*}
Vol_{\overline{g}_{i}(0)}(B_{\overline{g}_{i}(0)}(x,r))\geq \delta \omega_4 r^4\hspace{1cm}\forall x\in M_{i}, r\in [0,T_{i}^{-\frac{1}{4}}]
\end{equation*}
Hence, for each $\sigma\geq 1$ there exists $i_0(\sigma)\in\mathbb{N}$ so that
\begin{equation}\label{noncolscale}
Vol_{\overline{g}_{i}(0)}(B_{\overline{g}_{i}(0)}(x,r))\geq \delta \omega_4 r^4\hspace{1cm}\forall x\in M_{i}, r\in (0,\sigma]
\end{equation}
for all $i\geq i_0(\sigma)$. Now let $\lambda \in (0,1)$ be fixed. This constant will be made explicit below. Using \eqref{Volopenset} we obtain for $i\geq i_0(\sigma,\lambda, \delta)$
\begin{align}\label{noncolscaleII}
\begin{split}
&\left[Vol_{\overline{g}_{i}(1)}(B_{\overline{g}_{i}(0)}(x,\lambda \sigma))\right]^{\frac{1}{2}}\geq \left[Vol_{\overline{g}_{i}(0)}(B_{\overline{g}_{i}(0)}(x,\lambda \sigma))\right]^{\frac{1}{2}}- C\left(\frac{1}{i}\right)^{\frac{1}{2}}\\
\stackrel{\eqref{noncolscale}}{\geq}&\left[\delta\omega_4 (\lambda \sigma)^4\right]^{\frac{1}{2}}-C\left(\frac{1}{i}\right)^{\frac{1}{2}}=\left[(1-(1-\delta))\omega_4 \lambda^4 \sigma^4\right]^{\frac{1}{2}}- C\left(\frac{1}{i}\right)^{\frac{1}{2}}\\
\stackrel{\hphantom{\eqref{noncolscale}}}{\geq} & \left[(1-2(1-\delta))\omega_4 \lambda^4 \sigma^4\right]^{\frac{1}{2}}
\end{split}
\end{align} 
where the last estimate does not use that $i_0$ depends on $\sigma$, because, in order to choose $i_0\in\mathbb{N}$ large enough
one may fix $\sigma=1$ at first. Afterwards, one may multiply the inequality by $\sigma^2$. Since $\sigma\geq 1$, the desired estimate follows.

It is our intention to prove that
\begin{equation}\label{ballincl}
B_{\overline{g}_{i}(0)}(x,\lambda \sigma)\subseteq B_{\overline{g}_{i}(1)}(x,\sigma)\hspace{1cm}\forall i\geq i_0(\sigma,\lambda, \delta),\, \forall x\in M_{i}
\end{equation}
Before proving this, we demonstrate that this fact implies a contradiction.

For each $i\in\mathbb{N}$ we choose a point $p_{i}\in M_{i}$ satisfying 
\begin{equation*}
inj_{\overline{g}_{i}(1)}(M_{i},p_{i})=inj_{\overline{g}_{i}(1)}(M_{i})=1
\end{equation*}
 As above, using \cite[Corollary 1.5, p. 42]{streets2013long}, we may assume that there exists a subsequence of manifolds, again indexed by $i$, 
and a complete pointed $4$-manifold $(M_{\infty},p_{\infty})$ as well as a $1$-parametrized family of 
Riemannian metrics $(g_{\infty}(t))_{t\in [1/2,1]}$ on $M_{\infty}$ so that for each $t\in [1/2,1]$
\begin{equation*}
(M_{i},\overline{g}_{i}(t),p_{i})\stackrel{i\to\infty}{\longrightarrow} (M_{\infty},g_{\infty}(t),p_{\infty})
\end{equation*}
in the sense of $C^{\infty}$-local submersions. Using \cite[Theorem]{sakai1983continuity} we infer
\begin{equation}
\label{inj1}
 inj_{g_{\infty}(1)}(M_{\infty},p_{\infty})= 1
\end{equation}
Let $\zeta>0$ be equal to the non-collapsing parameter
in \cite[Gap Lemma 3.1, p. 440]{anderson1990convergence} which is denoted by $"\epsilon"$ in that work and only depends on the space dimension $n=4$.
We assume $\delta\in (0,1)$ and $\lambda\in (0,1)$ to be close enough to $1$ so that 
\begin{equation}\label{lamdelcl1}
(1-2(1-\delta))\lambda^4 \geq 1-\zeta
\end{equation}
 Assumed \eqref{ballincl} is valid, then for each 
for $i\geq i_0(\sigma,\lambda, \delta)$ we obtain the following estimate
\begin{align*}
Vol_{\overline{g}_{i}(1)}(B_{\overline{g}_{i}(1)}(p_{i},\sigma)) \stackrel{\eqref{ballincl}}{\geq} Vol_{\overline{g}_{i}(1)}(B_{\overline{g}_{i}(0)}(p_{i},\lambda \sigma))\stackrel{\eqref{noncolscaleII} / \eqref{lamdelcl1}}{\geq}(1-\zeta)\omega_4 \sigma^4
\end{align*}
and finally, as $i\in\mathbb{N}$ tends to infinity
\begin{align*}
Vol_{g_{\infty}(1)}(B_{g_{\infty}(1)}(p_{\infty},\sigma)) \geq (1-\zeta)\omega_4 \sigma^4\hspace{0.5cm}\forall \sigma\geq 1
\end{align*}
Then \cite[Gap Lemma 3.1, p. 440]{anderson1990convergence} implies that $(M_{\infty},g_{\infty}(1))$ is isometric 
to $(\mathbb{R}^4,g_{can})$ which contradicts \eqref{inj1}.

Hence, in order to prove the existence result and the validity of \eqref{eq:Rminj}, it remains to prove \eqref{ballincl}. From here on we do not write the subindex $i\in\mathbb{N}$.
 The following considerations shall be understood with $i\in\mathbb{N}$ fixed. That means that $p$ is one of the points $p_{i}$ 
and $\overline{g}(t)$ is the metric $\overline{g}_{i}(t)$ on $M=M_{i}$ with 
the same index. Let
\begin{equation}\label{ycontainedinball}
y\in B_{\overline{g}(0)}(p,\lambda \sigma)
\end{equation}
be an arbitrary point. As in the proof of  Lemma \ref{forwest} we construct a suitable forward-geodesic: Let
\begin{align*}
\beta:=\min_{t\in [0,1]}{\beta_t}>0
\end{align*}
where
\begin{equation*}
\beta_t:=\beta(4,diam_{\overline{g}(t)}(M),f_3(M,\overline{g}(t)),inj_{\overline{g}(t)}(M))
\end{equation*}
is chosen according to Lemma \ref{tube}. Next, using Lemma \ref{exforw}, we construct a $\beta$-forward-geodesic connecting $p$ and $y$ which is 
denoted by $(\xi_t)_{t\in [0,1]}$.  Hence, we have a finite set of geodesics $\left( \xi_{j S}\right)_{j\in \{0,..., \lfloor\frac{1}{S}\rfloor\}}$ which are parametrized proportional to arc length, i.e.:
\begin{equation*}
|\dot{\xi}_{jS}|_{g(jS)}\equiv d(p,y,jS)\text{ for all }j\in \{0,..., \lfloor\frac{1}{S}\rfloor\}
\end{equation*}
Furthermore, for each $j\in \{0,..., \lfloor\frac{1}{S}\rfloor\}$ let
\begin{align*}
 \varphi_{j }: [0, d(p,y,jS)]\longrightarrow [0,1]\\
\varphi(s)=\frac{s}{d(p,y,jS)}
\end{align*}
and let
\begin{align*}
 \gamma_{t}:=\xi_{j S}\circ\varphi_{j S}\text{ for each }  t\in[jS, (j+1)S)\cap [0,1]
\end{align*}
Applying the same argumentation as in the proof of Lemma \ref{forwest} we ensure that for each  $j\in \{0,..., \lfloor \frac{1}{S}\rfloor\}$ 
 and $t\in[jS, (j+1)S)\cap [0,1]$ the tubular neighborhood $D(\gamma_{t},\rho_t)$ is foliated by $\left(D(\gamma_{t}(s),\rho_t)\right)_{s\in [0,d(p,y,jS)]}$ 
where
\begin{equation*}
\rho_t:=\mu \min\left\{inj_{\overline{g}(t)}(M),f_3(M_,\overline{g}(t))^{-\frac{1}{2}} \right\} 
\end{equation*}
and the differential of the projection map satisfies \eqref{dpi}. Here $\mu>0$ is chosen fixed but also compatible to \cite[Lemma 2.7, p. 268]{streets2016concentration}.
We want to give a controlled lower bound on $\rho_t$. The curvature decay estimate from \eqref{eq:Rminjscaled} together with \eqref{boundcurv} implies for each $m\in\{1,2,3\}$:
\begin{equation}
\left\Vert \nabla^m Rm_{g(t)} \right\Vert_{L^{\infty}(M,g(t))}\leq C(m) t^{-\frac{2+m}{4}}\text{ for all }t\in (0,1]
\end{equation}
From this, we infer
\begin{equation*}
 f_3(M,g(t))\leq C t^{-\frac{1}{2}}\text{ on }(0,1]
\end{equation*}
Combining this estimate with the injectivity radius estimate from \eqref{eq:Rminjscaled}, we obtain, as in the proof of Lemma \ref{forwest} 
\begin{equation*}
\rho_t\geq \mu \left\{ t^{\frac{1}{4}}, C^{-\frac{1}{2}} t^{\frac{1}{4}} \right\}\geq   \mu \min\{1, C^{-\frac{1}{2}}\} t^{\frac{7}{24}}=: R t^{\frac{7}{24}}=:r_t  
\end{equation*}
we also obtain the estimate
\begin{align}\label{changeoflength}
\begin{split}
 \frac{d}{dt}{L(\gamma_t,t)} \leq & C_2 R^{-\frac{3}{2}} t^{-\frac{7}{16}}\left(\int_{M}{|\text{grad }\mathcal{F}_{\overline{g}(t)}|^2\, dV_{\overline{g}(t)}}\right)^{\frac{1}{2}}L^{\frac{1}{2}}(\gamma_{jS},t)\\
&+C_2 R t^{-\frac{23}{24}} L(\gamma_{jS},t)
\end{split}
\end{align}
on $(jS,(j+1)S)\cap [0,1)$ where $j\in\{1,...,\lfloor \frac{1}{S}\rfloor\}$. Now we assume that
\begin{align*}
&j_0:=\\
&\min\left\{j\in\{1,...,\lfloor \frac{1}{S}\rfloor\} \, |\, \exists t\in [j S,(j+1)S)\cap (0,1]\text{ s. th. } L(\gamma_{t},t)= \sigma \right\}
\end{align*}
exists, and  let
\begin{align*}
t_0:= \sup\left\{t\in [j_0 S,(j_0+1)S)\cap (0,1]\, |\,  L(\gamma_{\tau},\tau)\leq  \sigma\ \forall \tau\in [j_0 S,t]   \right\}
\end{align*}
Then, for each $j\in\{0,...,j_0\}$ and $t\in (jS,(j+1)S)\cap (0,t_0)$ estimate \eqref{changeoflength} implies   
\begin{align*}
 \frac{d}{dt}{L(\gamma_t,t)} \leq & \sigma  \left[C_2 R^{-\frac{3}{2}} t^{-\frac{7}{16}}\left(\int_{M}{|\text{grad }\mathcal{F}_{\overline{g}(t)}|^2\, dV_{\overline{g}(t)}}\right)^{\frac{1}{2}}
+ C_2 R t^{-\frac{23}{24}}\right] \\
\end{align*}
and consequently 
\begin{align*}
&d(p,y,t)-d(p,y,jS)\\
\leq &L(\gamma_{t},t)-L(\gamma_{jS},jS)\\
\leq &\sigma  C_2 R^{-\frac{3}{2}}\int_{jS}^t{ s^{-\frac{7}{16}}\left(\int_{M}{|\text{grad }\mathcal{F}_{\overline{g}(s)}|^2\, dV_{\overline{g}(s)}}\right)^{\frac{1}{2}}\, ds}+\sigma  C_2 R \int_{jS}^t{s^{-\frac{23}{24}}\,ds }
\end{align*}
In particular, for each $j\in\{0,...,j_0-1\}$ we infer
\begin{align*}
&d(p,y,(j+1)S)-d(p,y,jS)\\
\leq &\sigma  C_2  R^{-\frac{3}{2}}\int_{jS}^{(j+1)S}{ s^{-\frac{7}{16}}\left(\int_{M}{|\text{grad }\mathcal{F}_{\overline{g}(s)}|^2\, dV_{\overline{g}(s)}}\right)^{\frac{1}{2}}\, ds}+\sigma  C_2 R \int_{jS}^{(j+1)S}{s^{-\frac{23}{24}}\,ds }
\intertext{and}
&L(\gamma_{t_0},t_0)-d(p,y,j_0S)\\
\leq &L(\gamma_{t_0},t_0)-L(\gamma_{j_0S},j_0S)\\
\leq &\sigma C_2 R^{-\frac{3}{2}}\int_{j_0S}^{t_0}{ s^{-\frac{7}{16}}\left(\int_{M}{|\text{grad }\mathcal{F}_{\overline{g}(s)}|^2\, dV_{\overline{g}(s)}}\right)^{\frac{1}{2}}\, ds}+\sigma  C_2 R \int_{j_0 S}^{t_0}{s^{-\frac{23}{24}}\,ds }
\end{align*}
and finally
\begin{align*}
& L(\gamma_{t_0},t_0)-d(p,y,0)\\
\leq & L(\gamma_{t_0},t_0)-d(p,y,j_0S)+ \sum_{j=0}^{j_0-1}{[d(p,y,(j+1)S)-d(p,y,jS)]}\\
\leq &\sigma \left[C_2 R^{-\frac{3}{2}}\int_{0}^{t_0}{ s^{-\frac{7}{16}}\left(\int_{M}{|\text{grad }\mathcal{F}_{\overline{g}(s)}|^2\, dV_{\overline{g}(s)}}\right)^{\frac{1}{2}}\, ds}+ C_2 R \int_{0}^{t_0}{s^{-\frac{23}{24}}\,ds }\right]\\
\leq &\sigma \left[C_2 R^{-\frac{3}{2}}\int_{0}^{1}{ s^{-\frac{7}{16}}\left(\int_{M}{|\text{grad }\mathcal{F}_{\overline{g}(s)}|^2\, dV_{\overline{g}(s)}}\right)^{\frac{1}{2}}\, ds}+ C_2 R \int_{0}^{1}{s^{-\frac{23}{24}}\,ds }\right]\\
\leq &\sigma C_2 R^{-\frac{3}{2}}\left(\int_{0}^{1}{s^{-\frac{7}{8}}\, ds} \right)^{\frac{1}{2}}\left(\int_{0}^{1}{ \int_{M}{|\text{grad }\mathcal{F}_{\overline{g}(s)}|^2\, dV_{\overline{g}(s)}}\, ds}\right)^{\frac{1}{2}}\\
&\hspace{2cm}+\sigma C_2 R \int_{0}^{1}{s^{-\frac{23}{24}}\, ds} \\
\leq &\sigma C_3 R^{-\frac{3}{2}}\left(\int_{0}^{1}{s^{-\frac{7}{8}}\, ds} \right)^{\frac{1}{2}}\left(\int_{0}^{1}{ \int_{M}{|\text{grad }\mathcal{G}_{\overline{g}(s)}|^2\, dV_{\overline{g}(s)}}\, ds}\right)^{\frac{1}{2}}\\
&\hspace{2cm}+\sigma C_2 R \int_{0}^{1}{s^{-\frac{23}{24}}\, ds}\\ 
\leq &\sigma \left[C_4 R^{-\frac{3}{2}}\left(\int_{0}^{1}{ \int_{M}{|\text{grad }\mathcal{G}_{\overline{g}(s)}|^2\, dV_{\overline{g}(s)}}\, ds}\right)^{\frac{1}{2}}+ C_4 R \right] \\
\leq &\sigma C_4 R^{-\frac{3}{2}}\mathcal{G}^{\frac{1}{2}}_{\overline{g}(0)}+\sigma C_4 R \\
=&\sigma \left[ C_4 R^{-\frac{3}{2}}\mathcal{G}^{\frac{1}{2}}_{g(0)}+ C_4 R \right]
\end{align*}
Together with \eqref{ycontainedinball} we obtain
\begin{equation*}
L(\gamma_{t_0},t_0)< \sigma\left[ \lambda+C_4 R^{-\frac{3}{2}}\mathcal{G}^{\frac{1}{2}}_{g(0)}+ C_4 R \right]
\end{equation*}
Throughout, we may assume that $R>0$ is small enough compared to $C_4>0$ and $\lambda>0$ in order to ensure that 
\begin{equation*}
C_4 R\leq \frac{1-\lambda}{2}
\end{equation*} and we may assume that $i\in\mathbb{N}$ is chosen large enough, so that $\mathcal{G}_{g(0)}=\mathcal{G}_{g_{i}}\leq\frac{1}{i}$ is small enough compared to $\lambda>0$, $R(\lambda)>0$ and $C_4>0$ so that 
\begin{equation*}
C_4 R^{-\frac{3}{2}}\mathcal{G}^{\frac{1}{2}}_{g(0)}\leq \frac{1-\lambda}{2}
\end{equation*}
Hence, we have $L(\gamma_{t_0},t_0)<\sigma$, which contradicts $L(\gamma_{t_0},t_0)=\sigma$. This implies that $L(\gamma_t,t)<\sigma$ is valid for each $t\in [0,1]$ and consequently $d(p,y,1)<\sigma$. This finishes the proof of \eqref{ballincl}. 

We have proved the existence time estimate as well as the curvature decay estimate and the injectivity radius growth estimate.

It remains to show the diameter estimate \eqref{diamestimate}.
The argumentation is based on \cite[p. 281]{streets2016concentration} but we are in a different situation.
Let $x,y\in M$ so that $d(x,y,1)=diam_{g(1)}(M)$. As above, there exists $\beta>0$, $S>0$ and a family of curves $(\gamma_t)_{t\in [0,T]}$ so that 
\begin{itemize}
\item for each $j\in \left\{0,...,\left\lfloor\frac{T}{S}\right\rfloor \right\}$
\begin{align*}
\gamma_{jS}: [0,d(x,y,jS)]\longrightarrow M
\end{align*}
is a unit-speed length minimizing geodesic
\item for each $j\in \left\{0,...,\left\lfloor\frac{T}{S}\right\rfloor  \right\}$ and $t\in [jS, (j+1)S)\cap [0,T]$ 
the curve
\begin{equation*}
\gamma_{t}: [0,d(x,y,jS)]\longrightarrow M
\end{equation*}
satisfies
\begin{align*}
L(\gamma_t,t)\leq d(x,y,t)+\beta 
\end{align*}
\item  for each $j\in \left\{0,...,\left\lfloor\frac{T}{S}\right\rfloor  \right\}$ and $t\in [jS, (j+1)S)\cap [0,T]$ the tubular  neighborhood $D(\gamma_{t},r_t)$ is foliated by 
$\left( D(\gamma_{t}(s),r_t\right)_{s\in [0,d(x,y,jS)]}$
where
\begin{equation*}
r_t:=R t^{\frac{7}{24}}:=  \mu \min\{1, C^{-\frac{1}{2}} \} t^{\frac{7}{24}}
\end{equation*}
Furthermore, the projection map $\pi$ satisfies \eqref{dpi}, i.e.
\begin{equation*}
|d \pi |\leq 2\text{ for all } x\in D(\gamma,r_t)
\end{equation*}
\end{itemize}

Using these conditions we obtain \eqref{changeoflength}, i.e.:
\begin{align*}
 \frac{d}{dt}{L(\gamma_t,t)} \leq & C_2 R^{-\frac{3}{2}} t^{-\frac{7}{16}}\left(\int_{M}{|\text{grad }\mathcal{F}_{g(t)}|^2\, dV_{\overline{g}(t)}}\right)^{\frac{1}{2}}L^{\frac{1}{2}}(\gamma_{jS},t)\\
&+C_2 R t^{-\frac{23}{24}} L(\gamma_{jS},t)
\end{align*}
on $(jS,(j+1)S)\cap [0,T)$ where $j\in\{1,...,\lfloor \frac{T}{S}\rfloor\}$. In this situation we assume that
\begin{align*}
&j_0:=\\
&\min\left\{j\in\{1,...,\lfloor \frac{1}{S}\rfloor\} \, |\, \exists t\in [j S,(j+1)S)\cap (0,T]\text{ s. th. } L(\gamma_{t},t)= 2(1+D) \right\}
\end{align*}
exists, and we define
\begin{align*}
t_0:= \sup\left\{t\in [j_0 S,(j_0+1)S)\cap (0,T]\, |\,  L(\gamma_{\tau},\tau)\leq  2(1+D) \ \forall \tau\in [j_0 S,t]  \right\}
\end{align*}
Thus, for each $j\in\{0,...,j_0\}$ we obtain
\begin{align*}
 \frac{d}{dt}{L(\gamma_t,t)} \leq &  C_3 R^{-\frac{3}{2}} t^{-\frac{7}{16}}\left(\int_{M}{|\text{grad }\mathcal{F}_{g(t)}|^2\, dV_{g(t)}}\right)^{\frac{1}{2}}(1+D)^{\frac{1}{2}} \\
& + C_3 R t^{-\frac{23}{24}}(1+D) 
\end{align*}
 on $(jS,(j+1)S)\cap (0,t_0)$. From this, we infer
\begin{align*}
& L(\gamma_{t_0},t_0)-d(x,y,0)\\
\leq &L(\gamma_{t_0},t_0)-d(x,y,j_0S)+ \sum_{j=0}^{j_0-1}{[d(x,y,(j+1)S)-d(x,y,jS)]}\\
\leq &(1+D)C_3 \left[R^{-\frac{3}{2}}\int_{0}^{t_0}{ s^{-\frac{7}{16}}\left(\int_{M}{|\text{grad }\mathcal{F}_{g(s)}|^2\, dV_{g(s)}}\right)^{\frac{1}{2}}\, ds}+R \int_{0}^{t_0}{s^{-\frac{23}{24}}\,ds } \right]\\
\leq &(1+D) C_3 \left[R^{-\frac{3}{2}}\int_{0}^{1}{ s^{-\frac{7}{16}}\left(\int_{M}{|\text{grad }\mathcal{F}_{g(s)}|^2\, dV_{g(s)}}\right)^{\frac{1}{2}}\, ds}+ R \int_{0}^{t_0}{s^{-\frac{23}{24}}\,ds } \right]\\
\leq &(1+D) C_3 R^{-\frac{3}{2}}\left(\int_{0}^{1}{s^{-\frac{7}{8}}\, ds} \right)^{\frac{1}{2}}\left(\int_{0}^{1}{ \int_{M}{|\text{grad }\mathcal{F}_{g(s)}|^2\, dV_{g(s)}}\, ds}\right)^{\frac{1}{2}}\\
&\hspace{2cm}+(1+D)  C_3 R \int_{0}^{t_0}{s^{-\frac{23}{24}}\, ds} \\
\leq &(1+D)C_4 R^{-\frac{3}{2}}\left(\int_{0}^{1}{s^{-\frac{7}{8}}\, ds} \right)^{\frac{1}{2}}\left(\int_{0}^{1}{ \int_{M}{|\text{grad }\mathcal{G}_{g(s)}|^2\, dV_{g(s)}}\, ds}\right)^{\frac{1}{2}}\\
&\hspace{2cm}+(1+D) C_4 R \int_{0}^{t_0}{s^{-\frac{23}{24}}\, ds}\\ 
\leq &(1+D)C_4 \left[ R^{-\frac{3}{2}}\left(\int_{0}^{1}{s^{-\frac{7}{8}}\, ds} \right)^{\frac{1}{2}}\mathcal{G}^{\frac{1}{2}}_{g(0)}+R \int_{0}^{1}{s^{-\frac{23}{24}}\, ds} \right]\\
\leq &(1+D)C_5 \left[ R^{-\frac{3}{2}}\mathcal{G}^{\frac{1}{2}}_{g(0)}+ R \right]< 1+D
\end{align*}
Here, we have assumed that $\mathcal{G}^{\frac{1}{2}}(g(0))$ and $R>0$ are sufficiently small with respect to universal constants.
Finally, we obtain
\begin{equation*}
L(\gamma_{t_0},t_0)<d(x,y,0)+1+D=D+1+D<2(1+D)
\end{equation*}
contradicting $L(\gamma_{t_0},t_0)=2(1+D)$. This shows, that we have $diam_{g(t)}(M)\leq 2(1+D)$ for all $t\in [0,T]$.
\end{proof}

This existence result allows us to prove the following diffeomorphism finiteness result:
\begin{corollary}\label{diffeofin}
Let $D,\Lambda>0$. There exists $\epsilon(\Lambda)>0$ and a universal constant $\delta\in(0,1)$ so that there are only finitely many diffeomorphism types of closed Riemannian $4$-manifolds $(M,g)$ satisfying
\begin{align*}
diam_g(M)&\leq D\\
\Vert Rm_g \Vert_{L^2(M,g)}&\leq \Lambda\\
Vol_g(B_g(x,r))&\geq \delta \omega_4 r^4\hspace{1cm}\forall x\in M,\, r\in [0,1]\\
\Vert \mathring{Rc}_g \Vert_{L^2(M,g)}&\leq \epsilon
\end{align*}
\end{corollary}
\begin{proof}
Suppose there exists a sequence of Riemannian $4$-manifolds $(M_i,g_i)_{i\in\mathbb{N}}$ satisfying the desired properties but
the elements in this sequence are pairwise not diffeomorphic. Using Theorem \ref{existenceII} we may smooth out each of these manifolds, then
we may apply \cite[Theorem 2.2, pp. 464-466]{anderson1989ricci} at a fixed later time point which yields a contradiction.
\end{proof}

\begin{proof}[Proof of Theorem \ref{th:1aa}]
The proof is nearly analogous to the proof of Theorem \ref{th:1} but the argumentation is slightly different. Throughout, using Corollary \ref{diffeofin},  we assume that $M_i=M$ for all $i\in\mathbb{N}$, 
applying  Theorem \ref{th:4.2}, we may assume, that for each $i\in\mathbb{N}$ the $L^2$-flow on $M$ with initial data $g_i$ exists on $[0,T]$ and satisfies  \eqref{eq:2.6-}, \eqref{eq:2.6}, \eqref{eq:2.6a} and \eqref{eq:2.6b}. Using Corollary \ref{GHclose}, we choose a monotone decreasing sequence $(t_j)_{j\in\mathbb{N}}\subseteq (0,1]$ converging to zero, so that
\begin{equation*}
d_{GH}((M,g_i),(M,g_i(t_j)))<\frac{1}{3j}\hspace{1cm}\forall i,j\in\mathbb{N}
\end{equation*}
 \eqref{eq:2.6} and \eqref{boundcurv} together imply 
\begin{equation}
\left\Vert \nabla^m Rm_{g_i(t_j)} \right\Vert_{L^{\infty}(M,g_i(t_j))}\leq C(m)  t_j^{-\frac{2+m}{4}}\hspace{1cm}\forall i,j\in\mathbb{N}
\end{equation}
for each $m\in \mathbb{N}$, \eqref{eq:2.6a} implies
\begin{equation*}
inj_{g_i(t_j)}(M)\geq  t_j^{\frac{1}{4}}\hspace{1cm}\forall i,j\in\mathbb{N}
\end{equation*}

Applying the same argumentation as in the proof of Theorem \ref{th:1}  we infer 
\begin{equation*}
v_0(\delta)\leq Vol_{g_i(t_j)}{(M)}\leq V_0(D,\Lambda)
\end{equation*}
for all $i,j\in\mathbb{N}$. We want to point out that $\delta>0$ only depends on the space dimension which is constant. Using the
 flow convergence result in \cite[Corollary 1.5, p. 42]{streets2013long} on each time interval $[t_{j+1},t_j]$, starting with $t_0$, we obtain
a subsequence $(M_{i(j,k)},g_{i(j,k)}(t_j))_{k\in\mathbb{N}}$ as well as a family of Riemannian manifolds $(M,g_{\infty, j}(t))_{t\in [t_{j+1},t_j]}$ so that for each $t\in [t_{j+1},t_j]$ the sequence of Riemannian manifolds 
$(M,g_{i(j,k)}(t))_{k\in\mathbb{N}}$ converges smoothly to $(M,g_{\infty, j}(t))$ and the family of manifolds $(M_{\infty, j},g_{\infty, j}(t))_{t\in [t_{j+1},t_j]}$ is 
also a solution to the $L^2$-flow in the sense of \cite[Corollary 1.5, p. 42]{streets2013long}. Since $\mathcal{G}_{g_i(t)}\leq\mathcal{G}_{g_i}\leq \frac{1}{i}$ for all $i\in\mathbb{N}$, we conclude that $\mathcal{G}_{g_{\infty, j}(t)}=0$ for all $t\in [t_{j+1},t_j]$. 
Hence, at infinity, the metric does not change along the interval $[t_{j+1},t_j]$, which means that the manifold $(M_{\infty, j},g_{\infty, j}(t_j))=(M_{\infty, j},g_{\infty, j}(t_{j+1}))=:(M,g)$ is an Einstein manifold. Inductively, we obtain for
 each $j\in\mathbb{N}$ a sequence $(M_{i(j,k)},g_{i(j,k)}(t_j))_{k\in\mathbb{N}}$ that is a subsequence from $(M_{i(j-1,k)},g_{i(j-1,k)}(t_j))_{k\in\mathbb{N}}$, so that the sequence $(M_{i(j,k)},g_{i(j,k)}(t_j))_{k\in\mathbb{N}}$ converges to the Einstein manifold $(M,g)$. Using the same diagonal choice as in the Proof of Theorem \ref{th:1}, we infer that  there exists a subsequence of $(M_i,g_i)_{i\in\mathbb{N}}$ that also converges in the Gromov-Hausdorff topology to $(M,g)$.
\end{proof}

\appendix
\section{Auxilary Results}

In this paragraph we present some results which we have used in this work. Most of them are quoted from \cite{streets2016concentration}.

\begin{lem}
Let $(M^n,g(t))_{t\in [t_1,t_2]}$ be a smooth family of Riemannian manifolds and let $\gamma: [0,L]\longrightarrow M$ be a smooth curve. Then we have the estimates:
\begin{align}
\label{ap:1}
\left|\frac{d}{d t} L(\gamma,t)\right| &\leq \int_{\gamma}{{|g'(t)|_{g(t)}}\, d\sigma_t} \\
\label{ap:1a}
\left|\log\left(\frac{|v|^2_{g(t_2)}}{|v|^2_{g(t_1)}}\right)\right| &\leq \int_{t_1}^{t_2}{{\left\Vert  g'(t)  \right\Vert_{L^{\infty}(M,g(t))}}\, dt } \hspace{1cm}\forall v\in TM\\
\label{ap:2}
\left| \frac{\partial }{\partial t} \left| \nabla_{\dot{\gamma}}   \dot{\gamma} \right|_{g(t)} ^2\right|&\leq \left| g' \right|_{g(t)}  \left| \nabla_{\dot{\gamma}}   
\dot{\gamma} \right|_{g(t)} ^2+C(n) |\dot{\gamma}|_{g(t)} ^2  \left| \nabla_{\dot{\gamma}}   \dot{\gamma} \right|_{g(t)}   \left|\nabla g'\right|_{g(t)}
\end{align}
on $M\times (t_1,t_2)$.
\end{lem}
\begin{proof}
Using a unit-speed-parametrization of $\gamma$ we infer \eqref{ap:1}. Estimate \eqref{ap:1a} is proved in \cite[14.2 Lemma, p. 279]{hamilton1982three}. Estimate \eqref{ap:2}
is stated in \cite[p. 271]{streets2016concentration}.
\end{proof}

A simple calculation shows the following scaling behavior of the quantity $f_k$ (cf. Definition \ref{fk}). 

\begin{lem}\label{fkscaling}
Let $(M^n,g)$ be a closed Riemannian manifold, $k\in\mathbb{N}$,  $x\in M$ and $c>0$. Then we have the following equality
\begin{equation}\label{fkscalingeq}
f_k(x,cg)=c^{-1}f_k(x,g) 
\end{equation}
\end{lem}

From the definition of the gradient in \cite[4.10 Definition, p. 119]{besse2007einstein} we obtain:

\begin{lem}\label{energydifferencelem}
Let $(M^n,g(t))_{t\in[0,T]}$ be a smooth solution to the flow given in \eqref{eq:2.5} then we have:
\begin{equation}\label{eq:2.7}
\int_{0}^t{\int_{M}{|\text{grad }\mathcal{F}_{g(s)}|^2\, d{V_{g(s)}}}ds}=\mathcal{F}(g(0))-\mathcal{F}(g(t))
\end{equation}
for all $t\in [0,T]$.
\end{lem}

In particular, we can see that the energy $\mathcal{F}(g(t))$ is monotone decreasing under the flow given in \eqref{eq:2.5}, and 
\begin{equation}\label{eq:2.7a}
\int_{0}^t{\int_{M}{|\text{grad }\mathcal{F}_{g(s)}|^2\, dV_{g(s)}}\,ds}\leq \epsilon
\end{equation}
for all $t\in [0,T]$ under the assumption that $\mathcal{F}(g_0)\leq \epsilon$

\begin{theorem}{\label{derestimates}}(\cite[Lemma 2.11, p. 269]{streets2016concentration})
Fix $m,n\geq 0$. There exists a constant $C(n,m)>0$ so that if $(M^n,g(t))_{t\in [0,T]}$ is a complete solution to the $L^2$-flow satisfying
\begin{equation}\label{boundcurv}
\sup_{t\in[0,T]}{t^{\frac{1}{2}}\left\Vert Rm_{g(t)} \right\Vert_{L^{\infty}(M,g(t))}}\leq A
\end{equation}
then for all $t\in (0,T]$,
\begin{equation}\label{derbound}
\left\Vert \nabla^m Rm_{g(t)} \right\Vert_{L^{\infty}(M,g(t))}\leq C  \left((A+1)t^{-\frac{1}{2}}\right)^{1+\frac{m}{2}}
\end{equation}
\end{theorem}

\begin{lem}\label{L2volestimate}
Let $M^4$ be a closed Riemannian manifold and $(M,g(t))_{t\in [0,T]}$ be a solution to the $L^2$-flow. We have the following estimates
\begin{equation}\label{VolInvariant}
Vol_{g(t)}(M)=Vol_{g(0)}(M)\text{ for all }t\in (0,T]
\end{equation}
and
\begin{align}\label{Volopenset}
\begin{split}
Vol_{g(t)}(U)^{\frac{1}{2}}=Vol_{g(0)}(U)^{\frac{1}{2}} -Ct^{\frac{1}{2}} \left(\int_0^t{\int_{U}{|\text{grad }\mathcal{F}_{g(s)}|_{g(s)}^2\, dV_{g(s)}}\, ds}\right)^{\frac{1}{2}} \\
\text{ for all }t\in (0,T]\text{ and }U\subseteq M \text{ open} 
\end{split}
\end{align}
\end{lem}

\begin{proof}
The equation \eqref{VolInvariant} is a special case of the first equation in \cite[p. 44]{streets2013long}. Furthermore
\begin{align*}
&\left[Vol_{g(t)}(U)\right]^{\frac{1}{2}}-\left[Vol_{g(0)}(U)\right]^{\frac{1}{2}}\\
=&\int_0^t{\frac{d}{ds}{\left[Vol_{g(s)}(U)\right]^{\frac{1}{2}}}\, ds}
=\frac{1}{2}\int_0^t{\frac{\frac{d}{ds}{Vol_{g(s)}(U)}}{{\left[Vol_{g(s)}(U)\right]^{\frac{1}{2}}}}\, dt}\\
=&-\frac{1}{4}\int_0^t{\frac{\int_{U}{\text{tr}_{g(s)}\text{ grad }\mathcal{F}_{g(s)}}\, dV_{g(s)}}{{\left[Vol_{g(s)}(U)\right]^{\frac{1}{2}}}}\, ds}\\
\geq &-\frac{1}{4}\int_0^t{\frac{\left(\int_{U}{|\text{tr}_{g(s)}\text{ grad }\mathcal{F}_{g(s)}|^2}\, dV_{g(s)}\right)^{\frac{1}{2}}}{{\left[Vol_{g(s)}(U)\right]^{\frac{1}{2}}}}{\left[Vol_{g(s)}(U)\right]^{\frac{1}{2}}}\, ds}\\
\geq & -C \int_0^t{\left(\int_{U}{|\text{grad }\mathcal{F}_{g(s)}|_{g(s)}^2\, dV_{g(s)}}\right)^{\frac{1}{2}}\, ds}\\
\geq & -Ct^{\frac{1}{2}} \left(\int_0^t{\int_{U}{|\text{grad }\mathcal{F}_{g(s)}|_{g(s)}^2\, dV_{g(s)}}\, ds}\right)^{\frac{1}{2}} \\
\end{align*} 
\end{proof}

\begin{lem}\label{coarealem}(\cite[Lemma 2.8, p. 268]{streets2016concentration})
Let $(M,g)$ and $(N,h)$ be smooth Riemannian manifolds and let $F: M\longrightarrow N$ be a smooth submersion. Furthermore, let $\phi: M\longrightarrow [0,\infty)$ be a smooth function, then one has:
\begin{equation}\label{coarea}
\int_{M}{\phi\, dV_g}=\int_{y\in N}{\int_{x\in F^{-1}(y)}{\frac{\phi(x)}{NJac\, F(x)}\, dF^{-1}(y)}\, dV_h}
\end{equation}
where $NJac\, F(x)$ is the determinant of the derivative restricted to the orthogonal complement of its kernel. This quantity is also called  ``normal Jacobian''. 
\end{lem}

\begin{lem}\label{ricflatimplboundcurv}
Let $n\in\mathbb{N}$, $\iota>0$ and let $(M^n,g)$ be a complete $n$-dimensional Riemannian manifold
such that the following is true
\begin{align*}
Rc_g&\equiv 0 \\
\left\Vert Rm_g \right\Vert_{L^{\infty}(M^n,g)}&< \infty\\
inj_g(M)&\geq \iota
\end{align*}
then
\begin{equation*}
\left\Vert Rm_g \right\Vert_{L^{\infty}(M^n,g)}\leq C(n,\iota).
\end{equation*}
\end{lem}
\begin{proof}
We argue by contradiction. Suppose this statement would be wrong, then we could find a sequence of complete $n$-dimensional 
Ricci-flat manifolds $(M_i,g_i)_{i\in\mathbb{N}}$ so that 
\begin{align*}
inj_{g_i}(M_i)&\geq \iota
\intertext{and}
\left\Vert Rm_{g_i}\right\Vert_{L^{\infty}(M_i,g_i)}&=C_i
\intertext{where}
\lim_{i\to\infty}{C_i}&=\infty
\end{align*}
We construct a blow-up sequence as follows:  for each $i\in\mathbb{N}$ let
\begin{equation*}
h_i:=C_i \cdot g_i
\end{equation*}
so that

\begin{align*}
inj_{h_i}(M_i)&\geq \sqrt{C_i} \iota
\intertext{and}
\left\Vert Rm_{h_i}\right\Vert_{L^{\infty}(M_i,h_i)}&=1
\end{align*}
For each $i\in\mathbb{N}$ we choose a fixed point $p_i\in M_i$, so that $|Rm_{h_i}(p_i)|_{h_i}\geq \frac{1}{2}$. Using $Rc_{h_i}\equiv 0$, the first equation on
\cite[p. 461]{anderson1989ricci} or \cite[7., 7.1. Theorem, p. 274]{hamilton1982three} implies
\begin{equation}\label{DeltaRmEinst}
\Delta_{h_i} Rm_{h_i}=Rm_{h_i}\ast Rm_{h_i}
\end{equation}
and consequently
\begin{equation*}
\left\Vert \Delta_{h_i} Rm_{h_i}\right\Vert_{L^{\infty}(M_i,h_i)}\leq K(n)
\end{equation*}
Furthermore, from \cite[Lemma 1]{hildebrandt1977existence}, we obtain uniform $C^0$-bounds on the metrics $(h_i)_{i\in\mathbb{N}}$ in normal coordinates. Hence, an iterative application
of the theory of linear elliptic equations of second order to \eqref{DeltaRmEinst}, following the arguments of \cite[p. 478, second paragraph]{anderson1989ricci}, we obtain uniform higher order estimates, i.e.:
\begin{equation*}
 \left\Vert \nabla_{h_i}^k Rm_{h_i}\right\Vert_{L^{\infty}(M_i,h_i)}\leq K(n,k)
\end{equation*}
for all $i,k\in\mathbb{N}$. Hence, \cite[Theorem 2.2, pp. 464-466]{anderson1989ricci} implies that there exists a subsequence $(M_i,g_i,p_i)_{i\in\mathbb{N}}$ that converges in the pointed $C^{k,\alpha}$-sense, where 
$k\in\mathbb{N}$ is arbitrary, to a smooth manifold $(X,h,p)$ satisfying
\begin{equation*}
|Rm_{h}(p)|_{h}\geq \frac{1}{2}
\end{equation*}
and, using \cite[Theorem]{sakai1983continuity}
\begin{equation*}
inj_{h}(X,p)=\infty 
\end{equation*}
An iterative application of \cite[Theorem 2]{cheeger1971splitting} implies that $(X,h,p)=(\mathbb{R}^n,g_{euc},0)$ which yields a contradiction.
\end{proof}

\section{Notation}
Here, we give an overview of the notation that we are using in this work. Sometimes it is clear that
a quantity depends on a certain metric. In this situation we often omit the dependency in the notation, i.e. $Rm_g=Rm$ for instance.
\begin{itemize}
\item For $i\in \{1,...,n\}$ $\partial_i=\frac{\partial}{\partial x^i}$ denotes a coordinate vector in a local coordinate system
\item $g_{ij}$ is a Riemannian metric in a local coordinate system and $g^{ij}$ is the inverse of the Riemannian metric
\item $dV_g=dV$ is the volume form induced by a Riemannian metric $g$
\item $Vol_g(\cdot)=Vol(\cdot)$ is the $n$-dimensional volume of a set in a Riemannian manifold $(M,g)$
\item $dA_g=dA$ is the $n-1$-dimensional volume form induced by a Riemannian metric $g$
\item $Area_g(\cdot)=Area(\cdot)$ is the $n-1$-dimensional volume of a set in a Riemannian manifold $(M,g)$
\item $\omega_n$ is the euclidean volume of a euclidean unit ball 
\item $Rm_g=Rm$ is the Riemannian curvature tensor. As in \cite{streets2008gradient}, in local coordinates,
 the sign convention is consistent with \cite[p. 5]{chow2006hamilton}, i.e. $R_{ijkl}=R^m_{ijk} g_{ml}$. 
\item $Rc_g=Rc$ is the Ricci tensor
\item $R_g=R$ is the scalar curvature
\item $\frac{\partial}{\partial t}g=g'$ is the time derivative of the metric
\item $\text{grad }\mathcal{F}_{g}$ is the gradient of the functional $\mathcal{F}_{g}$ with respect to $g$ (cf. \cite[Chapter 4, 4.10 Definition, p. 119]{besse2007einstein})
\item $\mathring{Rc_g}=\mathring{Rc}$ is the traceless Ricci tensor, i.e.: $\mathring{Rc_g}=Rc_g-\frac{1}{n}R g$
\item $^g \nabla T$=$\nabla T$ is the covariant derivative of a tensor $T$ with respect to g
\item $^g \nabla^m T$=$\nabla^m T$ is the covariant derivative of order $m$
\item $\langle T,S \rangle_g$=$\langle T,S \rangle$ is the inner product of two tensors 
\item $|T|_g=|T|$ is the norm of a tensor, i.e. $|T|_g:=\sqrt{\langle T,T\rangle}_{g}$
\item $diam_g(\cdot)=diam(\cdot)$ is the diameter of a set in a Riemannian manifold
\item $inj_{g}(M,x)$ is the injectivity radius in a point of a Riemannian manifold 
\item $inj_{g}(M)$ is the injectivity radius of a Riemannian manifold 
\item $d_g(x,y)=d(x,y)$ is the distance between the points $x$ and $y$ in a Riemannian manifold
\item $B_d(x,r)=B(x,r)$ is the ball of radius $r>0$ around $x$ in a metric space
\item $d_g$ is the metric which is induced by a Riemannian metric $g$
\item $B_g(x,r)=B_{d_g}(x,r)$ is a metric ball in a Riemannian manifold 
\item $d(x,y,t)$ is the distance between the points $x$ and $y$ in a Riemannian manifold $(M,g(t))$
\item $L(\gamma,t)$ is the length of a curve $\gamma$ in a Riemannian manifold $(M,g(t))$
\item The notation $d\sigma$, which occurs in an integral like $ \int_{\gamma}{ \left|\text{grad }\mathcal{F}\right|\, d\sigma  }$, refers
to the integration with respect to arc length
\item $D(\gamma(t),r)$ / $D(\gamma,r)$ is a normal disc around a point in a curve $\gamma$ / a (normal) tube around a curve $\gamma$ with radius $r$ (cf. Definition \ref{tubedefi})
\item $f_k(x,g)$ / $f_k(M,g)$ is introduced in Definition \ref{fk}
\item $d \pi$ denotes the push forward and $|d \pi |$ denotes the operator norm of the push forward of the projection map in the context of Theorem \ref{tube}
\item $\Gamma$ denotes the local bilinear form in Definition \ref{Gamma}, $|\Gamma|$ is the norm of this bilinear form which is also introduced in Definition \ref{Gamma}
\end{itemize}

The following definition is based on \cite[(1), p. 261]{kaul1976schranken}
\begin{defi}\label{Gamma}
Let $(M^n,g)$ be a smooth Riemannian manifold $p\in M$, $U\subseteq M$ a star-shaped neighborhood around $p$, and $\varphi: U\longrightarrow V$
a normal chart centered at $p$, then for each $q\in U$ we define a symmetric, bilinear map $\Gamma $ as follows:
\begin{align*}
\Gamma : T_q M\times T_q M&\longrightarrow T_q M \\
(u,v)&\mapsto \Gamma_{ij}^k u^i v^i \partial_k
\end{align*}
and $|\Gamma|$ is defined to be the smallest value $C>0$ so that
\begin{align*}
|\Gamma(u,v)|_g\leq C |u|_g |v|_g
\end{align*}
for all $u,v\in T_p M$.
\end{defi}

\bibliographystyle{alphaurl}
\bibliography{Conv4mani}

\end{document}